\documentclass[11pt]{article}
\usepackage{amssymb}
\usepackage{mathrsfs}
\usepackage{latexsym,amsmath,amssymb,amsfonts,epsfig,graphicx,cite,psfrag}
\usepackage{eepic,color,colordvi,amscd}
\usepackage{ebezier}
\usepackage{verbatim}
\usepackage{subfigure}
\usepackage{algorithm}
\usepackage{algorithmic}
\usepackage{amsthm}
\usepackage{comment}
\usepackage{color}
\definecolor{blue}{rgb}{0,0,0.9}
\definecolor{red}{rgb}{0.9,0,0}
\definecolor{green}{rgb}{0,0.9,0}

\usepackage{hyperref}
\usepackage{longtable}
\usepackage{cleveref}
\usepackage{bbding}
\theoremstyle{plain}
\newtheorem{theo}{Theorem}[section]
\newtheorem{claim}[theo]{Claim}

\newtheorem{lem}[theo]{Lemma}
\newtheorem{coro}[theo]{Corollary}

\newtheorem{prop}[theo]{Proposition}

\theoremstyle{definition}
\newtheorem{defi}[theo]{Definition}

\theoremstyle{remark}
\newtheorem{rem}[theo]{Remark}
\newtheorem{assump}{Assumption}
\def\pf{\noindent {\it Proof.\quad}}

\def\Tr{\mathrm{Tr}}

\def\<{\left\langle}
\def\inprod#1#2{\big\langle#1,\,#2\big\rangle}
\def\norm#1{\|#1\|}

\def\>{\right\rangle}

\def\M{\mathcal{M}}
\def\T{\mathcal{T}}
\def\D{{\rm D}}
\def\DD{{\rm Diag}}
\def\A{\mathcal{A}}

\def\Ka{\mathcal{K}_{n,r}}

\def\R{\mathbb{R}}

\def\W{\mathcal{W}}

\def\Re{\mathcal{R}}

\def\Q{\mathcal{Q}}
\def\N{\mathcal{N}}

\def\S{\mathbb{S}}
\def\s{{\rm S}}

\def\mr{{\rm mr}}

\def\rr{{\rm rank}}
\def\dist{{\rm dist}}
\def\dd{{\rm diag}}
\def\P{{\rm Proj}}
\def\OB{{\rm OB}}

\def\Ron{\texttt{Round}}
\def\({\left(}
\def\){\right)}

\let\svthefootnote\thefootnote
\newcommand\blankfootnote[1]{%
	\let\thefootnote\relax\footnotetext{#1}%
	\let\thefootnote\svthefootnote%
}

\begin{document}
	\title{ A feasible method for solving an SDP relaxation of the quadratic knapsack problem}
	\author{Tianyun Tang
\thanks{Department of Mathematics, National
         University of Singapore, Singapore
         119076 ({\tt ttang@u.nus.edu}).
         }, \quad 
	 Kim-Chuan Toh\thanks{Department of Mathematics, and Institute of 
Operations Research and Analytics, National
         University of Singapore, 
       Singapore
         119076 ({\tt mattohkc@nus.edu.sg}).  The research of this author is supported
by the Ministry of Education, Singapore, under its Academic Research Fund Tier 3 grant call (MOE-2019-T3-1-010).}
	 }

	\date{\today}
	\maketitle
%%%%%%%%%%%%%%%%%%%%%%%%%%%Extension%%%%%%%%%%%%%%%%%%%%%%%%%%%%

\begin{abstract}
In this paper, we consider an SDP relaxation of the quadratic knapsack problem (QKP). 
After applying  low rank factorization, we get a non-convex problem, whose feasible region is an algebraic variety with {certain good geometric properties which we analyse.} We derive a rank condition under which these two formulations are equivalent. This rank condition is much weaker than the classical rank condition if the coefficient matrix has certain special structures. We also prove that under an appropriate rank condition, the non-convex problem has no spurious local minima without {assuming linearly independent constraint qualification}. We design a feasible method that can escape from non-optimal non-regular points. Numerical experiments are conducted to verify the high efficiency and robustness of our algorithm as compared to other solvers. In particular, our algorithm is able to solve a one-million dimensional sparse SDP problem accurately in about 20 minutes on a modest computer.
\end{abstract}

\section{Introduction.}\label{SecIntro}

\subsection{An SDP relaxation of quadratic knapsack problem.}\label{SubSDP}

The binary quadratic knapsack problem (QKP) was introduced in \cite{QKP} by Gallo et al.  as follows: 
\begin{equation}\label{QKP}
{\rm (QKP)}\ \max\left\{ x^\top C x:\ a^\top x\leq \tau,\ x\in \{0,1\}^n\right\},  \notag
\end{equation}
where $C\in \S^n$ is the profit matrix, $a\in \R^n$ is the positive weight vector and $\tau>0$ is the knapsack capacity. We consider the nontrivial case where $n>1.$ We also make the following natural assumption on the data $(a,\tau)$.
\begin{assump}\label{ass1}

The weight vector $a\in \R^n$ and capacity $\tau\in \R$ satisfy the condition that for any $i\in [n]$, $0<a_i<\tau$ and $e^\top a>\tau,$ where $e$ is the vector of all ones.
\end{assump}

The problem (QKP) has many applications such as {very large-scale integration (VLSI)} and compiler design \cite{VLSI,compiler}. When $C$ is a diagonal matrix, (QKP) reduces to a binary linear knapsack problem. (QKP) is NP-hard because it can be reduced to the clique problem (which is closely related to the densest $k-$subgraph problem)  by choosing $\tau=k\in \mathbb{N}^+,$ $a$ to be the all ones vector and $C$ to be the adjacency matrix of an undirected graph. Like many combinatorial optimization problems, (QKP) is extremely difficult to be solved exactly. Many SDP relaxations are provided in the literature to calculate an upper bound of (QKP). In \cite{SQK3}, Pisinger discussed 4 types of SDP relaxations of (QKP) whose optimal values are ordered as $U^1_{\rm HRW}\geq U^2_{\rm HRW}\geq U^3_{\rm HRW}\geq U^4_{\rm HRW}$ respectively. {These SDP models were originally provided by Helmberg et al. in \cite{SQK2}. However, {the authors} only mentioned the SDP model with optimal value $U^3_{\rm HRW}$ for a special case of (QKP) and it was later generalized by Pisinger in \cite{SQK3}.} Among them, the first three models have the same number of constraints and the last one has multiple inequality constraints. However, it has been shown in the numerical experiments of \cite{SQK2} that even though the first three SDP models are simpler than the last one, $U^2_{\rm HRW}$ is already quite close to $U^4_{\rm HRW}.$ Therefore, in this paper, we only consider the third SDP model with both tightness and simplicity. We write it as follows:
\begin{equation}\label{SQK}
{\rm (SQK)}\ \min\left\{ -\<C,X\>:\ a^\top Xa-\tau a^\top x\leq 0,\ \dd\(X\)=x,\ Y:=\begin{pmatrix} 1&x^\top \\ x&X\end{pmatrix}\in \S^{n+1}_+\right\}. \notag
\end{equation}
{While the original problem is a maximization problem, we add a minus sign to the objective function to transform it into a minimization problem.} Note that (SQK) is an SDP problem with the matrix variable $Y$ of size $(n+1)\times (n+1).$ The SDP constraint $Y\in \S^{n+1}_+$ is equivalent to $X-xx^\top\in \S^n_+$ by the Schur complement lemma.
%which is used in \cite{SQK2,SQK3}. 
For notational convenience, we will sometimes use $(1,x^\top; x, X)$ to denote the matrix 
$\begin{pmatrix} 1&x^\top \\ x&X\end{pmatrix}.$

\subsection{Our contributions.}\label{SubOC}
In this paper, instead of solving (SQK), we consider the following problem that transfers the inequality constraint in (SQK) into an affine constraint:
\begin{equation}\label{SQKE}
{\rm (SQKE)}\ \min\left\{ -\<C,X\>:\ a^\top Xa-\tau a^\top x= 0,\ \dd\(X\)=x,\ Y:=\begin{pmatrix} 1&x^\top \\ x&X\end{pmatrix}\in \S^{n+1}_+\right\}. \notag
\end{equation}
Although in general, (SQKE)  is not equivalent to (SQK), we {can easily} construct an optimal solution of (SQK) from the solution of (SQKE) in section~\ref{Rela}. The feasible region of (SQKE) is nonempty because it has a feasible solution 
%$Y:=\begin{pmatrix}1&0_{1\times n}\\0_{n\times 1}&0_{n\times n}\end{pmatrix}$. 
$Y:= (1,0_{1\times n}\,;\,0_{n\times 1},0_{n\times n})$. 
Note that from the last two constraints in (SQKE), we have that for any $i\in [n],$ 
%%$\begin{pmatrix} 1&x_i\\x_i&x_i\end{pmatrix}\succeq 0.$ 
$( 1,x_i\,;\,x_i,x_i)\succeq 0.$
This implies that $0\leq x_i\leq 1,$ which in addition implies that every entry of $X$ is inside the interval $[-1,1].$ Therefore, the feasible region of (SQKE) is compact.  
(SQKE) is a typical linear SDP problem, which can be solved with guaranteed convergence by various well developed solvers such as SDPT3 \cite{SDPTTT,SDPT3}, MOSEK \cite{mosek}, and SDPNAL \cite{SDPNALp,SDPNAL}. Those solvers are efficient and accurate for problems of moderate size. However, if $n$ is large, (SQKE) will be challenging to solve by the aforementioned solvers due to excessive computing cost and memory demand. In order to solve (SQKE) efficiently, we have conducted both theoretical and computational studies, which are listed as follows.

\subsubsection{A modified BM factorization.}\label{SubBM}
In order to overcome the large dimensionality, we consider the following low rank version of (SQKE),
\begin{equation}\label{SQKELR}
{\rm (SQKELR)}\ \min\left\{ -\inprod{C}{RR^\top} :\ \norm{a^\top R}^2-\tau a^\top Re_1= 0,\ \dd\big(RR^\top\big)=Re_1,\ R\in \R^{n\times r}\right\}, \notag
\end{equation}
where $e_1\in \R^r$ is such that its first entry is 1 and all other entries are zero. The above factorization model comes from the fact that for any 
%%$Y:=\begin{pmatrix} 1&x^\top \\ x&X\end{pmatrix}\in \S^{n+1}_+$ 
$Y:=(1,x^\top \,;\, x,X)\in \S^{n+1}_+$ 
of rank $r$, it can be factorized as $\begin{pmatrix} e_1 &R^\top\end{pmatrix}^\top\begin{pmatrix} e_1 &R^\top\end{pmatrix}$ for some $R\in \R^{n\times r}.$ Note that our factorization is different from the classical Burer-Monteiro factorization (see \cite{BM1,BM2}) because we fix the first row of the low rank matrix to be $e_1^\top.$ By doing so, we can avoid dealing with the constraint $Y_{11}=1$ because it is automatically satisfied. Another advantage of this factorization is that we can prove that there are only finitely many non-regular points in the feasible region of (SQKELR). However, this is not the case for the usual 
Burer-Monteiro factorization. The reason is that for any matrix $R\in \R^{n\times r}$ and orthogonal matrix $Q\in \R^{r\times r},$ $RR^\top=RQQ^\top R^\top.$ This implies that $R$ is equivalent to $RQ$ and one non-regular point will implies infinite many non-regular points.

\subsubsection{An algebraic variety with good geometric properties.}\label{SubAV}

There are many algorithms and approaches for solving the low rank formulation of linear SDP problems. A general solver is SDPLR by Burer and Monteiro (see \cite{BM1,BM2}), which use an augmented Lagrangian method to solve the factorized SDP problems. Another approach is based on a feasible method, which strictly preserves the constraints. A feasible method has the advantage that it can terminate in advance when the primal feasibility is more important than the optimality of a solution. A famous application of feasible methods is the low rank SDP relaxation of max-cut problem (see \cite{Boumal1,BMmaxcut,BM1,WenYin}). Because a feasible method usually relies on the manifold structure of the feasible region, in section~\ref{Geom}, we will analyze the geometric properties of the feasible region of (SQKELR) i.e.,

\begin{equation}\label{Kamani}
\Ka:=\left\{ R\in \R^{n\times r}:\ \norm{a^\top R}^2-\tau a^\top Re_1= 0,\ 
\dd\big(RR^\top\big)=Re_1\right\}.
\end{equation}

We will show that for a generic data $\(a,\tau\)\in \R^n\times \R$ that satisfies Assumption~\ref{ass1}, $\Ka$ is smooth everywhere except for the trivial point $R=0$. We will also study the local geometric properties of $\Ka$ at non-regular points, which is useful for our theoretical analysis and algorithmic design.

\subsubsection{A new rank bound for (SQKE) to be equivalent to (SQKELR).}\label{SubR}

An interesting problem to investigate is when is (SQKE) equivalent to (SQKELR). {Note that there are only $m:=n+2$ affine constraints and one of the affine constraints of (SQKE) is $Y_{11}=1.$} The compactness of the nonempty feasible region of (SQKE) implies that it has an optimal solution of rank $\leq \lceil\sqrt{2(n+2)}\rceil$  (see, e.g., \cite{rank1,rank2,rank3}). Thus, when $r\geq \lceil\sqrt{2(n+2)}\rceil,$ (SQKELR) and (SQKE) have the same optimal function value. Note that $\lceil\sqrt{2m}\rceil$ is the theoretical rank-bound, which only relies on the number of constraints of an SDP problem while ignoring any structural information of the coefficient data. In subsection~\ref{rankbound}, we will derive a new rank-bound that is related to the structure of the {profit} matrix $C.$ This new bound is much stronger than $\lceil\sqrt{2(n+2)}\rceil$ in the case of a binary linear knapsack problem. 

\subsubsection{(SQKELR) doesn't have spurious local minima.}\label{NoSp}

Another interesting problem to investigate is when does (SQKELR) have no spurious local minima. Boumal et al. did a series of works (see \cite{Boumal0,Boumal1,Boumal2}) to study when does the Burer-Monteiro formulation of a linear SDP problem have no spurious local minima. Their analysis is based on the assumption that the feasible region of the low rank SDP problem is a smooth manifold. In subsection~\ref{Secsta}, we will show that any second order stationary point of (SQKELR) is a global optimal solution under a suitable rank condition that is similar to the results in \cite{Boumal0,Boumal1,Boumal2}. Compared with their works, our result has the advantage that we don't need any regularity assumption.

\subsubsection{A Riemannian optimization approach that can handle non-regular points.}\label{RieAlg}
In section~\ref{Algorithm}, we will discuss the details of our algorithmic implementation. We will use the theoretical results in previous sections to design an algorithm that combines a Riemannian optimization method with a verification and escaping strategy at non-regular points. Our algorithm is similar to algorithm 1 in \cite{GEP} and both of them can stop after finite number of outer iterations. We also apply rounding to the output of our algorithm to get an integer solution that is feasible for (QKP). In section~\ref{Numerical}, we conduct numerical experiments to verify the efficiency of our algorithm compared with the solver SDPLR and another recent solver SketchyCGAL 
 developed by Yurtsever et al. in \cite{CGAL}. Note that all these three algorithms make use of the low rank property of (SQKE) but our algorithm is the only feasible method that exploit the geometric properties of (SQKELR). Our algorithm can be more than 100 times faster than the other two solvers for some high dimensional problems with $n\geq 5000$. For a large sparse profit matrix $C$ of size $n=10^6$, our algorithm can find an accurate optimal solution of (SQKE) within half an hour. We also test the behaviour of our algorithm with a rounding procedure for solving (QKP). We show that our algorithm can find a nearly optimal solution for (QKP) very efficiently.

\subsection{Notations.}\label{Notation}
In this paper, we often omit stating the dimension of a vector or matrix if it is already clear from the context. Without mentioning explicitly, $\|.\|$ always denote the matrix Frobenius norm. We denote the trace of a square matrix $X$ as $\Tr(X)$.
For two matrices with the same size, $\<A,B\>:=\Tr\(AB^\top\)$ is the matrix inner product; $e$ is the vectors of all ones; $e_1$ is the vector such that its first entry is 1 and all the other entries are zero; ${\rm S}^{r-1}:=\left\{ x\in \R^r:\ \|x\|=1\right\};$ $\S^n$ and $\S^n_+$ denote the space of $n\times n$ symmetric matrices and its subset of symmetric positive semidefinite matrices, respectively. For notational simplicity, we use $X\succeq 0$ to denote $X\in \S^n_+$. For $x\in \R^n,$ $\dd\(x\)\in \S^n$ is the diagonal matrix whose diagonal is given by $x.$ For $X\in \R^{n\times n}$, $\dd(X) \in \R^n$ is the diagonal vector of $X$, and $\DD(X) = \dd(\dd(X)).$ For $A,B\in \R^{m\times n}$, $A\circ B\in \R^{m\times n}$ is defined by $\(A\circ B\)_{ij}=a_{ij}b_{ij}.$
 
\subsection{Organization of the rest of the paper.}\label{Orga}
In section~\ref{Rela}, we discuss when is (SQK)  equivalent to (SQKE) and how to solve (SQK) via solving (SQKE). In section~\ref{Geom}, we study the geometric properties of  the algebraic variety $\Ka.$ In section~\ref{Equiv}, we discuss the equivalence between (SQKE) and (SQKELR), which includes a new rank-bound for (SQKELR). In section~\ref{Algorithm}, we design an algorithm to solve (SQKELR). In section~\ref{Numerical}, we conduct numerical experiments to compare our algorithm with SDPLR and SketchyCGAL. In section~\ref{Conclusion}, we give a brief conclusion. 

\section{Relation between (SQK) and (SQKE).}\label{Rela}

In this section, we consider how to construct an optimal solution of (SQK) by solving (SQKE). We have the following proposition.
\begin{prop}\label{equiv}
Suppose $(a,\tau)$ satisfy Assumption~\ref{ass1}. If every entry of $C$ is nonnegative, then (SQK) and (SQKE) have the same optimal function value.
\end{prop}

\begin{proof}
\pf
{Note that Assumption~\ref{ass1} implies that $\tau>0.$} We only have to prove that there exists an optimal solution 
%%$Y:=\begin{pmatrix} 1&x^\top \\ x&X\end{pmatrix}$ 
$Y:= (1,x^\top \,;\, x,X)$ 
of (SQK) such that $a^\top Xa-\tau a^\top x= 0.$ In order to prove this, assume that $a^\top Xa-\tau a^\top x<0$ for some optimal solution $Y$ of (SQK) because otherwise there is nothing to prove. Consider another matrix 
%$\widehat{Y}:=\begin{pmatrix} 1&\hat{x}^\top \\ \hat{x}&\widehat{X}\end{pmatrix}=ee^\top.$ 
$\widehat{Y}:=(1,\hat{x}^\top \,;\, \hat{x},\widehat{X})=ee^\top.$ 
From Assumption~\ref{ass1}, we have that 
\begin{equation}\label{May_20_1}
a^\top \widehat{X}a-\tau a^\top \hat{x}=\big(a^\top e\big)^2-\tau a^\top e>0.
\end{equation}
 Also, because every entry of $C$ is nonnegative, and every entry of $Y$ is inside the interval $[-1,1]$ as shown in the introduction. We have that 
 \begin{equation}\label{May_20_2}
 -\inprod{C}{\widehat{X}}\leq -\inprod{ C}{X}.
 \end{equation}
 Now, consider $Y_t:=t\widehat{Y}+(1-t)Y.$ From (\ref{May_20_1}), (\ref{May_20_2}) and the convexity of (SQK), we have that there exists $t\in (0,1)$ such that $Y_t$ is inside the feasible region of (SQKE) with its function value less than or equal to the function value at $Y.$ From the optimality of $Y,$ we have that $Y_t$ is also an optimal solution of (SQK). Thus, we have found an optimal solution of (SQK), which is also an optimal solution of (SQKE).
 \end{proof}

Proposition~\ref{equiv} requires that the profit matrix $C$ is nonnegative. However, the following proposition tells us that if $C$ is not nonnegative, {we} can still construct an optimal solution of (SQK) by solving (SQKE) and the following simpler problem.
\begin{equation}\label{SQKS}
{\rm (SQKS)}\ \min\left\{ -\<C,X\>:\ \dd\(X\)=x,\ \begin{pmatrix} 1&x^\top \\ x&X\end{pmatrix}\in \S^{n+1}_+\right\}. \notag
\end{equation}

\begin{prop}\label{equiv1}
Suppose $(a,\tau)$ satisfy Assumption~\ref{ass1}. If (SQK) and (SQKE) don't have the same optimal function value, then any optimal solution of (SQKS) is also an optimal solution of (SQK).
\end{prop}
\begin{proof}
\pf
Let $V,V_E,V_S$ be the optimal function values of (SQK), (SQKE) and (SQKS) respectively. It is easy to see that $V_E\geq V\geq V_S.$ Because (SQK) and (SQKE) have different optimal function values, we have that $V_E>V\geq V_S.$ Now, assume by contradiction that there exists an optimal solution 
%%$Y_S:=\begin{pmatrix} 1&x^\top \\ x&X\end{pmatrix}$ 
$Y_S:= (1,x^\top \,;\, x,X)$ 
of (SQKS) such that $a^\top Xa-\tau a^\top x> 0.$ Because $V_E>V,$ we have that there exists an optimal solution 
%%$Y:=\begin{pmatrix} 1&\hat{x}^\top \\ \hat{x}&\widehat{X}\end{pmatrix}$ 
$Y:=(1,\hat{x}^\top \,;\, \hat{x},\widehat{X})$ 
of (SQK) such that $a^\top \widehat{X}a-\tau a^\top \hat{x}<0.$ Now, define $Y_E^t=tY+(1-t)Y_S.$ We have that there exists $t\in (0,1)$ such that $Y_E^t$ is feasible for (SQKE). Also, from $V_E>V\geq V_S,$ we have that the function value of $Y_E^t$ is strictly smaller than $V_E.$ This contradicts to the fact that $V_E$ is the optimal function value of (SQKE).
\end{proof}

\medskip
In order to solve (SQK), we can first solve (SQKS) to get an optimal solution $Y.$ If $Y$ is feasible for (SQK), then $Y$ is also an optimal solution of (SQK). Otherwise if $Y$ is not feasible for (SQK), then by Proposition~\ref{equiv1}, (SQK) and (SQKE) must have the same optimal function value. In this case, we only have to solve (SQKE) to get an optimal solution of (SQK). Note that (SQKS) is easy to solve because its low rank factorization only has the nonlinear constraint $\dd\(RR^\top\)=Re_1,$ which is equivalent to that $\dd\( \(2R-ee_1^\top\)\(2R-ee_1^\top\)^\top\)=e.$ Thus, solving the low rank version of (SQKS) is equivalent to minimizing a quadratic function on the following oblique manifold
\begin{equation}
\OB_{n,r}:=\left\{R\in \R^{n\times r}:\ \dd\(RR^\top\)=e\right\}.\notag
\end{equation}
This problem can be solved efficiently by a widely used toolbox called manopt (see \cite{manopt}). Since (SQK) can always be reduced to (SQKE), in the sequel, we will focus on analysing the property of (SQKE) instead of (SQK).

%%%%%%%%Algebraic variety%%%%%%%%%%%%%%%%%%%%%%%%%%
\section{Geometric properties of $\Ka$.}\label{Geom}
In this section, we focus on the geometric properties of the feasible region of (SQKELR), $\Ka$
defined in (\ref{Kamani}). Note that $\Ka$ is not empty because $0_{n\times r}\in \Ka.$ We call this point a trivial point because it is of little use in practice. Actually, this point corresponds to the case that we don't select any item in (QKP). 
%%%%%Non-regular points%%%%%%%%%%%%%%%%%%%%%%%%%%
\subsection{Regularity of $\Ka$.}\label{RegKa}
From now on, we say that a point $R\in \Ka$ is regular if the linear independent constraint qualification (LICQ) holds at $R.$ Otherwise we say that $R$ is non-regular. Note that LICQ is a sufficiently condition for a point to be smooth on an algebraic variety. If $R\in \Ka$ is regular, then $\Ka$ is a smooth manifold locally around $R$ and its tangent cone is a linear space, which is also called the tangent space. It is easy to check that for any $(a,\tau)$ satisfying Assumption~\ref{ass1},
 $0_{n\times r}$ is a non-regular point. The following proposition characterizes all the non-regular points on $\Ka.$

\begin{prop}\label{regular}
Suppose $(a,\tau)$ satisfy Assumption~\ref{ass1}. The non-regular points of $\Ka$ are given by
\begin{equation}\label{nreg}
\N_{n,r}:=\{0_{n\times r}\}\sqcup\left\{ ve_1^\top:\ v\in \{0,1\}^n,\ a^\top v=\tau\right\}.
\end{equation}
\end{prop}
\begin{proof}
\pf
Since the trivial point $0_{n\times r}$ is non-regular, we only have to characterize the set of non-trivial and non-regular points of $\Ka.$ Assume that $R\in \Ka$ is non-trivial and non-regular. Because the LICQ doesn't hold at $R,$ we have that there exists nonzero $\(\mu,\lambda\)\in \R^n\times \R$ such that 
\begin{equation}\label{May_21_1}
\dd(\mu)\(2R-ee_1^\top\)+\lambda a\( 2a^\top R-\tau e_1^\top\)=0.
\end{equation}
Because $R\in \Ka,$ the constraint $\dd\(RR^\top\)=Re_1$ implies that $2R-ee_1^\top\in \OB_{n,r}.$ That is, every row of $2R-ee_1^\top$ is on the unit sphere of $\R^{1\times r}.$ If there exists $i\in [n]$ such that $\mu_i=0,$ then from (\ref{May_21_1}) and the fact that every entry of $a$ is positive, we have that $\lambda \( 2a^\top R-\tau e_1^\top\)=0.$ This implies that $\dd(\mu)\(2R-ee_1^\top\)=0$ and so $\mu=0.$ Because $\(\mu,\lambda\)$ is nonzero, we have that $\lambda\neq 0.$ Thus, $\( 2a^\top R-\tau e_1^\top\)=0$ and so $a^\top R=\frac{1}{2}\tau e_1^\top.$ It follows that $\|a^\top R\|^2-\tau a^\top Re_1=-\frac{\tau^2}{4}<0,$ which contradicts to the fact that $R\in \Ka.$ Thus we conclude the fact that for any $i\in [n],$ $\mu_i\neq 0.$

Now from (\ref{May_21_1}), we have that 
\begin{equation}\label{May_21_2}
2R-ee_1^\top+\dd\(\mu\)^{-1}\lambda a\big( 2a^\top R-\tau e_1^\top\big)=0.
\end{equation}
 Because $2R-ee_1^\top\in \OB_{n,r}$ and $\dd\(\mu\)^{-1}\lambda a\( 2a^\top R-\tau e_1^\top\)$ is rank 1, we have that $2R-ee_1^\top=u b^\top$ for some $u\in \{-1,1\}^n$ and $b\in {\rm S}^{r-1}$ such that the first nonzero entry of $b$ is positive. From (\ref{May_21_2}), we have that
\begin{equation}\label{May_21_3}
b^\top=\alpha \(a^\top u b^\top + (a^\top e -\tau) e_1^\top\),
\end{equation}
for some nonzero normalization factor $\alpha\in \R.$ If $b$ is not parallel to $e_1$, then from (\ref{May_21_3}), we have that $a^\top e-\tau=0,$ which contradicts the fact that $a^\top e>\tau.$ Thus $b$ is parallel to $e_1$, and hence $b=e_1,$ {{since} $b\in{\rm S}^{r-1}$ and its first nonzero entry is positive.} Therefore,  we have that $R=\(\frac{e+u}{2}\)e_1^\top.$ Define $v:=\frac{e+u}{2},$ we have that $v\in \{0,1\}^n$ and $R=ve_1^\top.$ Because $R\in \Ka,$ we have that 
\begin{equation}\label{May_21_4}
0=\|a^\top R\|^2-\tau a^\top Re_1=\big(a^\top v\big)^2-\tau a^\top v.
\end{equation}
Because $R$ is non-trivial and every entry of $a$ is positive, we have that $a^\top v> 0.$ This together with (\ref{May_21_4}) implies that $a^\top v=\tau.$ Therefore, $R\in \N_{n,r}.$

For any nontrivial point $R=ve_1^\top\in \N_{n,r}.$ Because $v\in \{0,1\}^n$ and $a^\top v=\tau,$ we have that $R=ve_1^\top\in \Ka.$ Thus, we only have to check that $ve_1^\top$ is a non-regular point of $\Ka.$ This can be easily verified by choosing $\(\mu,\lambda\)=\( \dd(a)\(2v-e\),\frac{-1}{\tau}\)$ in (\ref{May_21_1}).
\end{proof}

\medskip
{From Proposition~\ref{regular}, one can see that the non-trivial non-regular points of $\Ka$ corresponds to the integer feasible solutions of (QKP) such that the knapsack upper bound are attained.} Proposition~\ref{regular} has the following corollary, which says that $\Ka$ is likely to be smooth everywhere except for the trivial point.

\begin{coro}\label{regular1}
For a generic $(a,\tau)\in \R^n\times \R$ that satisfies Assumption~\ref{ass1}, $\N_{n,r}=\{0_{n\times r}\}$.
\end{coro}
\begin{proof}
\pf
First, note that the set of vectors $(a,\tau)$ satisfying Assumption~\ref{ass1} is an open set in $\R^n\times \R$ with positive measure. From (\ref{nreg}), it is easy to see that the set of $(a,\tau)$'s that satisfies Assumption 1 and $\N_{n,r}\neq \{0_{n\times r}\}$ is the following set
\begin{equation}\label{nreg-W}
\W_{n,\tau}:=\cup_{v\in \{0,1\}^n}\left\{ (a,\tau)\in \R^n\times \R:\ a^\top v=\tau, \forall\ i\in [n]\ 0<a_i<\tau,\ e^\top a>\tau\right\}.
\end{equation}
{Note that the affine constraint $a^\top v=\tau$ implies that $\W_{n,\tau}$ is the union of finitely many sets of measure zero in $\R^n\times \R$.} Therefore, we have that for a generic $(a,\tau)\in \R^n\times \R$ that satisfies Assumption~\ref{ass1}, $\N_{n,r}=\{0_{n\times r}\}.$
\end{proof}

%%%%%%%%%Tangent cone%%%%%%%%%%%%%%%%%%%%%%%%%%%
\subsection{Local geometric properties at non-regular points of $\Ka$.}\label{GeoKa}

Although Corollary~\ref{regular1} implies that there are no non-trivial non-regular points of $\Ka$ with probability one with respect to the data $\(a,\tau\)$ satisfying Assumption~\ref{ass1}, for theoretical completeness, we will discuss the local geometric properties at non-regular points of $\Ka$ in this subsection. Before presenting our result, we need the definition of the tangent cone and second order tangent set. For any closed set $S\subset \R^n$ and $x\in S,$ the inner tangent cone (see \cite{Perturbation} section 2) $\T^i_S(x)$ is defined as follows:
\begin{equation}\label{tancone}
\T^i_S(x):=\left\{ h\in \R^n:\ \dist\(x+th,S\)=o(t),\ t\geq 0\right\},
\end{equation}
{where $\dist(x,S):=\inf_{y\in S}\|x-y\|.$} For any vector $h\in \T^i_S(x),$ the inner second order tangent set is defined as follows (see \cite{Perturbation} definition 3.28):
\begin{equation}\label{istan}
\T^{i,2}_S(x,h):=\left\{ w\in \R^n:\ {\dist}\Big(x+th+\frac{1}{2}t^2w,S\Big)=o(t^2),t\geq 0\right\}.
\end{equation}

Note that there are also other kinds of tangent cones and second order tangent sets (see section 2 and section 3 in \cite{Perturbation} for more details). However, in this paper, we only use the definition of inner tangent cone and inner second order tangent set. Thus, for convenience, we use $\T_S(x)$ and $\T^2_S(x,h)$ to denote $\T^i_S(x)$ and $\T^{i,2}_S(x,h)$ respectively.

\begin{prop}\label{pgct}
Suppose $(a,\tau)$ satisfies Assumption~\ref{ass1} and $r\geq 2.$ Let $R=ve_1^\top\in \N_{n,r},$ where $v\in \{0,1\}^n.$ Then the tangent cone of $\Ka$ at $R$ satisfies
\begin{equation}\label{May_22_4}
\widetilde{\T}:=\left\{ [0,H]:\< aa^\top-\sigma_v\tau\cdot\dd\( a\circ d\),HH^\top\>=0,\ H\in \R^{n\times (r-1)}\right\}\subset \T_{\Ka}\(R\),
\end{equation}
where $d=2v-e$ and $\sigma_v:=1$ if $v\neq 0,$ and $\sigma_v:=-1$ if $v=0.$ In addition, for any $h=[0,H]\in \widetilde{\T},$ $w=[-2\dd\(HH^\top\)\circ d,\,0]\in \T^2_{\Ka}\(R,h\).$
\end{prop}
\begin{proof}
\pf
For any $[0,H]\in \widetilde{\T},$ define the matrix valued function $R(t,W):(-\delta,\delta)\times \R^{n\times (r-1)}\rightarrow \R^{n\times r}$ such that
\begin{equation}\label{May_22_5}
R(t,W):=\left[ \( v-\frac{e}{2}\)\circ \sqrt{e-4\dd\big(\( tH+t^3W\)\( tH+t^3W\)^\top\big)}+\frac{e}{2},\ tH+t^3W\right],
\end{equation}
where the square-root is taken componentwise and $\delta>0$ is sufficiently small for the square root to be well-defined. {Note that the first column of $R(t,W)$ {is derived} from solving the equation $\dd\(RR^\top\)-Re_1=0.$} It is easy to check that $\dd\(R(t,W)R(t,W)^\top\)-R(t,W)e_1=0$. By using the Taylor expansion of the square root term of $R(t,W),$ we get
\begin{align}\label{May_22_6}
R(t,W)=&\Big[ v-\(t^2\dd\big(HH^\top\big)+2t^4 \dd\big(HW^\top\big)+t^4\dd^2\big(HH^\top\big)\)\circ d+t^6g(t,W),\notag \\
&\;\;tH+t^3W \Big],
\end{align}
where $g(t,W)$ is a smooth function. 
In order to get $R(t,W)\in \Ka,$ we only have to guarantee that 
\begin{equation}\label{May_22_7}
4\norm{ a^\top R(t,W)}^2-4\tau a^\top R(t,W)e_1=\norm{ 2a^\top R(t,W)-\tau e_1^\top}^2-\tau^2=0.
\end{equation}
Note that from the definition of $\N_{n,r}$ and $\sigma_v$, we have that $2a^\top v-\tau=\sigma_v \tau.$ By using (\ref{May_22_6}), we have that
\begin{align}\label{May_22_8}
&\| 2a^\top R(t,W)-\tau e_1^\top\|^2-\tau^2\notag \\
&=-4t^2\sigma_v \tau\cdot a^\top \big(\dd\big( HH^\top\big)\circ d\big)+4t^4\( a^\top \big(\dd\big( HH^\top\big)\circ d\big)\)^2
-8t^4\sigma_v\tau\cdot a^\top \big(\dd\big( HW^\top\big)\circ d\big)\notag \\[3pt]
&\quad -4t^4\sigma_v \tau\cdot a^\top \big(\dd^2\big(HH^\top\big)\circ d\big)
+8t^4 \inprod{a^\top H}{a^\top W} +4t^2\| a^\top H\|^2+t^6k\(t,W\)\notag \\[3pt]
&=8t^4 \<aa^\top H-\sigma_v \tau\cdot\dd\(a\circ d\) H,W \>
\notag \\
&\quad+4t^4\( \big( a^\top \big(\dd\big( HH^\top\big)\circ d\big)\big)^2
-\sigma_v \tau\cdot a^\top \big(\dd^2\big(HH^\top\big)\circ d\big)\)+t^6k(t,W),
\end{align}
where the second equality comes from that $\< aa^\top-\sigma_v\tau\cdot\dd\( a\circ d\),HH^\top\>=0.$ Also, $k(t,W)$ is a smooth function. 

\begin{claim}\label{Hsmo}
If $H\neq 0,$ then $aa^\top H-\sigma_v \tau\cdot\dd\(a\circ d\) H\neq 0.$
\end{claim}
\proof{\it Proof of Claim~\ref{Hsmo}.}
Assume by contradiction that $aa^\top H-\sigma_v \tau\cdot\dd\(a\circ d\) H= 0.$ Because every entry of $a$ is positive, we have that $H=\frac{ da^\top H}{\tau\sigma_v}.$ If $a^\top H=0,$ then $H=\frac{ da^\top H}{\tau\sigma_v}=0,$ which contradicts to the fact that $H\neq 0.$ Thus, $a^\top H\neq 0.$ Because $H=\frac{ da^\top H}{\tau\sigma_v},$ we have that $a^\top H=a^\top \frac{da^\top H}{\tau\sigma_v},$ which implies that 
 $a^\top d=\tau\sigma_v.$ If $v\neq 0,$ then $\sigma_v=1$, and from (\ref{nreg}), $a^\top v=\tau.$ This together with $a^\top d=\tau$ and $d=2v-e$ implies that $a^\top e=\tau,$ which contradicts to the fact that $a^\top e>\tau.$ If $v=0,$ then $\sigma_v=-1$ and $d=2v-e=-e.$ It follows that $a^\top e=-a^\top d=-\tau\sigma_v=\tau.$ Similarly, we get a contradiction.  $\renewcommand\qedsymbol{\ding{122}}$ \endproof

Now we define the following function $F:(-\delta,\delta)\times \R^{n\times (r-1)}\rightarrow \R$ such that
\begin{align}\label{May_22_9}
&F(t,W):=8\<aa^\top H-\sigma_v \tau\cdot\dd\(a\circ d\) H,W \>
\notag \\
&\qquad\qquad +4\Big( \big( a^\top \big(\dd( HH^\top)\circ d\big)\big)^2-\sigma_v \tau\cdot a^\top \big(\dd^2(HH^\top)\circ d\big)\Big)+t^2k(t,W).
\end{align}
If $H\neq 0,$ then from Claim~\ref{Hsmo}, there exists $W_0\in \R^{n\times (r-1)}$ such that $F(0,W_0)=0.$ Moreover, since $\nabla_W F(0,W_0)=8\(aa^\top H-\sigma_v \tau\cdot\dd\(a\circ d\) H\)\neq 0.$ From Lemma~\ref{implicit}, we have that there exists a continuously differentiable mapping $\widehat{W}:(-\epsilon,\epsilon)\rightarrow \R^{n\times (r-1)}$ such that $\widehat{W}(0)=W_0$ and $F(t,\widehat{W}(t))=0.$ This implies that $R(t,\widehat{W}(t))\in \Ka$ for any $t\in (-\epsilon,\epsilon).$ Because 
$\dist\big( R(t,\widehat{W}(t)),\left[ v-t^2\dd\(HH^\top\)\circ d,tH\right]\big)=O(t^3),$ we have that $h=[0,H]\in \T_{\Ka}\(R\)$ and 
$w=[-2\dd\(HH^\top\)\circ d,0]\in \T_{\Ka}^2\(R,h\).$ 
If $H=0,$ then it is obvious that $0=[0,0]\in \T_{\Ka}\(R\)$ and $w=[-2\dd\(HH^\top\)\circ d,0]=0\in \T_{\Ka}^2\(R,0\).$
\end{proof}

\medskip
\begin{rem}\label{complete}
In order to save space, we don't provide the complete local geometric properties at non-regular points of $\Ka$ in Proposition~\ref{pgct}. {The converse of (\ref{May_22_4}) is also correct and it can be proved similarly based on the Taylor expansion used the in the proof of Proposition~\ref{pgct}.} However, {the above} properties are already enough for our later theoretical analysis and algorithmic design.
\end{rem}

%%%%%%%%%%%%%%%%%%%%%%%%%%%%%%rank%%%%%%%%%%%%%%%%%%%%%%%%%%%%%%%%
\section{Equivalence between (SQKE) and (SQKELR).}\label{Equiv}
{In this section, we will study the equivalence between (SQKE) and (SQKELR). We will first derive a new rank condition under which (SQKE) and (SQKELR) have the same global optimal value. After that, we will discuss when a second order stationary point of (SQKLRE) is also a global optimal solution.}
%%%%%%%%%%%%%%%%%%%%%%%%%%%%%%Minimum rank%%%%%%%%%%%%%%%%%%%%%%%%%%%%%%%%
\subsection{A new rank-bound for (SQKELR).}\label{rankbound}
In this subsection, we discuss the equivalence between (SQKE) and (SQKELR). As we mentioned in the introduction, (SQKE) and (SQKELR) are equivalent when $r\geq \sqrt{2(n+2)}.$ Here, we will provide a new rank-bound which relies on the structure of the matrix $C.$ Before showing our main result, we present some background knowledge on the minimum rank problem of a graph. Consider a simple undirected graph $G$ with vertices set $V(G)=[n]$ and edges set $E(G)\subset \binom{[n]}{2}.$ We define the following set
\begin{equation}\label{QG}
\Q(G):=\left\{ A\in \S^{n}:\ \forall i\neq j,\ A_{ij}\neq 0\Leftrightarrow ij\in E(G)\right\},
\end{equation}
which is the set of symmetric matrices whose sparsity patterns are specified by the adjacency matrix of $G.$
The $minimum\ rank$ of $G$ (see \cite{MG2}) is defined as follows:
\begin{equation}\label{MG}
\mr(G):=\min\left\{ \rr\(A\):\ A\in \Q(G)\right\}.
\end{equation}
For a given graph $G$, determining $\mr(G)$ is a difficult problem that has received considerable attention (see \cite{MG1} for a survey). In the next theorem, we give a new rank-bound of (SQKELR) using $\mr(G).$ 
\begin{theo}\label{MGrank}
Consider the problem (SQKE), and $(a,\tau)$ satisfies Assumption~\ref{ass1}. Suppose $C\in \Q(G)$ for some graph $G$ and there exists another graph $G_1$ such that $V(G_1)=[n]$ and $E(G)\subset E(G_1).$ Then there is an optimal solution 
%%$Y^*=\begin{pmatrix}1&{x^*}^\top\\x^*&X^*\end{pmatrix}$ 
$Y^*=(1,{x^*}^\top\,;\,x^*,X^*)$
of (SQKE) such that $\rr\(Y^*\)\leq n+2-\mr(G_1).$
\end{theo}
\begin{proof}
\pf
Because $C\in \Q(G)$ and $E(G)\subset E(G_1),$ there exists a matrix sequence $\{C_k\}\subset \S^n$ such that for any $k\in \mathbb{N}^+$, $C_k\in \Q(G_1)$ and $\lim_{k\rightarrow \infty}\|C_k-C\|=0.$ Let (SQKE)$_k$ be the optimization problem obtained by replacing $C$ with $C_k$ in (SQKE). Suppose 
%%$Y_k=\begin{pmatrix}1&{x_k^*}^\top\\x_k^*&X_k^*\end{pmatrix}$ 
$Y_k=(1,{x_k^*}^\top\,;\,x_k^*,X_k^*)$ 
is an optimal solution of problem (SQKE)$_k.$ From the compactness of the feasible region of (SQKE), and $\lim_{k\rightarrow \infty}\|C_k-C\|=0,$ we have that any accumulation point of $\{Y_k\}$ is an optimal solution of (SQKE). Without loss of generality, we may assume that $Y_{i_k}\rightarrow Y^*$ for some convergent sequence $\{Y_{i_k}\}$. Thus, 
%%$Y^*=\begin{pmatrix}1&{x^*}^\top\\x^*&X^*\end{pmatrix}$ 
$Y^*=(1,{x^*}^\top\,;\,x^*,X^*)$ 
is an optimal solution of (SQKE). 

\begin{claim}\label{rankY}
For any $k\in \mathbb{N}^+,$ $\rr\(Y_{i_k}\)\leq n+2-\mr(G_1).$
\end{claim}
\proof{\it Proof of Claim~\ref{rankY}.}
Note that $\mr(G_1)\leq |G_1|-1=n-1$  (see observation 1.2 of \cite{MG1}). Assume by contradiction that there exists $k\in \mathbb{N}^+$ such that $\rr\(Y_{i_k}\)> n+2-\mr(G_1).$ Then we have that $\rr\(Y_{i_k}\)\geq 4.$ Let (SQKELR)$_{i_k}$ be the low rank formulation of (SQKE)$_{i_k}$ with $r=\rr\(Y_{i_k}\).$  Because $Y_{i_k}$ the an optimal solution of (SQKE)$_{i_k},$ we have that (SQKELR)$_{i_k}$ is equivalent to (SQKE)$_{i_k}.$ Moreover, consider the low rank factorization $Y_{i_k}=\begin{pmatrix} e_1 & R_{i_k}^\top\end{pmatrix}^\top
\begin{pmatrix} e_1 & R_{i_k}^\top\end{pmatrix}$
in (SQKELR)$_{i_k}.$ Then $R_{i_k}$ is an optimal solution of (SQKELR)$_{i_k}$ such that $\rr\(R_{i_k}\)\geq \rr\( Y_{i_k}\)-1\geq 3.$  Because $\rr\(R_{i_k}\)\geq 3$,
by Proposition~\ref{regular}, it satisfies the LICQ of (SQKELR)$_{i_k}$. Thus $R_{i_k}$ satisfies the following first order KKT condition of (SQKELR)$_{i_k}$
\begin{equation}\label{Mayy_22_1}
-2\(C_{i_k}+\dd\(\mu\)\)R_{i_k}+\mu e_1^\top-\lambda a\big(2a^\top R_{i_k}-\tau e_1^\top\big)=0,
\end{equation}
for some $\(\mu,\lambda\)\in \R^n\times \R.$ 
Rearranging (\ref{Mayy_22_1}), we get
\begin{equation}\label{Mayy_22_2}
\big(\mu,-2\(C_{i_k}+\dd\(\mu\)\)\big)\begin{pmatrix}e_1^\top \\R_{i_k}\end{pmatrix}=\lambda a\big(2a^\top R_{i_k}-\tau e_1^\top\big).
\end{equation}
Because $C_{i_k}\in \Q(G_1),$ we have that
\begin{equation}\label{Mayy_22_3}
\rr\Big( \big(\mu,-2\(C_{i_k}+\dd\(\mu\)\)\big)\Big)\geq 
\rr\Big( -2\(C_{i_k}+\dd\(\mu\)\)\Big)\geq \mr\(G_1\).
\end{equation}
Applying (\ref{Mayy_22_3}) in (\ref{Mayy_22_2}), we have that
\begin{align}\label{Mayy_22_4}
&1\geq \rr\( \big(\mu,-2\(C_{i_k}+\dd\(\mu\)\)\big)\begin{pmatrix}e_1^\top \\R_{i_k}\end{pmatrix} \)\notag \\
&\geq \rr\Big( \big(\mu,-2\(C_{i_k}+\dd\(\mu\)\)\big)\Big)+\rr\( \begin{pmatrix}e_1^\top \\R_{i_k}\end{pmatrix} \)-(n+1), \notag \\
&\geq \mr(G_1)+\rr\( \begin{pmatrix}e_1^\top \\R_{i_k}\end{pmatrix} \)-(n+1)
\end{align}
where the second inequality follows from Sylvester’s rank inequality (see 10.5 of \cite{matrixcook}). Therefore, 
$\rr\( Y_{i_k}\)=\rr\big( \begin{pmatrix}e_1\ R_{i_k}^\top\end{pmatrix}^\top \big)\leq n+2-\mr(G_1),$ which contradicts to the assumption that $\rr\(Y_{i_k}\)>n+2-\mr(G_1).$ $\renewcommand\qedsymbol{\ding{122}}$ \endproof

From Claim~\ref{rankY}, since $Y_{i_k}\rightarrow Y^*,$ we have that $\rr\(Y^*\)\leq  n+2-\mr(G_1).$
\end{proof}

\medskip
The rank-bound of $n+2-\rr(G_1)$ in Theorem~\ref{MGrank} is different from the rank-bound of $\sqrt{2(n+2)}$ because it is related to the structure of the underlying graph $G$ of matrix $C$. Many existing results in the literature for the minimum rank problem can be applied to estimate the number $n+2-\rr(G_1)$ (see \cite{MG1,MG2,MGn2}). A simple application of Theorem~\ref{MGrank} is on (SQKE) for a binary linear knapsack problem. In this case, $C$ is a diagonal matrix and its underlying graph $G$ is an empty graph. We have that $E(G)\subset E(P_n),$ where $P_n$ is a path with $n$ vertices. Because $\mr(P_n)=n-1$ (see observation 1.2 of \cite{MG1}), the rank-bound becomes $n+2-\rr(P_n)=3,$ which is much better than $\sqrt{2(n+2)}$ when $n$ is large. The SDP relaxation of a linear knapsack problem has been demonstrated to be tighter than its linear programming relaxation (see {lemma 2.8} in \cite{SQK2}). Our results additionally implies that {we} can use (SQKE) to get a better bound than linear programming relaxation with almost the same memory usage. Another interesting case is when the matrix $C$ is a banded matrix such that $C_{ij}=0$ if $|i-j|>k.$ In that case,  Theorem~\ref{MGrank} implies that (SQKE) has an optimal solution of rank $\leq 2+k.$

\subsection{Second order stationarity of a non-regular point of (SQKELR).}\label{Secsta}

In this subsection, we discuss when is a second order stationary point of (SQKELR) also a global optimal solution. Traditional optimality conditions are applicable only to regular points in the feasible region of an optimization problem (see Section 3 of \cite{Perturbation}). Thus, when $R$ is a regular point of (SQKELR), we will call $R$ a second order stationary point if it satisfies the first order and second order KKT conditions (see theorem 12.1 and 12.5 in \cite{numerical}). However, since $\Ka$ may have non-regular points, we need to define the second order stationarity for non-regular points of (SQKELR). From Proposition~\ref{pgct} on the local geometric properties of a non-regular point in $\Ka$, we have the following lemma.

\begin{lem}\label{esc}
Suppose $(a,\tau)$ satisfies Assumption~\ref{ass1} and $r\geq 2.$ Let $R=ve_1^\top\in \N_{n,r},$ where $v\in \{0,1\}^n$ and $d\in \{-1,1\}^n,\sigma_v\in \{-1,1\}$ are defined as in Proposition~\ref{pgct}. If $R$ is a local minimum of (SQKELR), then the optimal value of the following problem is nonnegative.
\begin{equation}\label{escpro}
\min\left\{ \inprod{\( 2\dd\( \(Cv\)\circ d\)-C\)H}{H} :\  
\begin{array}{l}
\<aa^\top-\sigma_v\tau\cdot\dd\(a\circ d\),HH^\top\>=0,\\[5pt]
 \|H\|^2=1,\ H\in \R^{n\times (r-1)}
\end{array} 
\right\}.
\end{equation}
\end{lem}
\begin{proof}
\pf
Assume by contradiction that there exists a feasible solution $H$ of (\ref{escpro}) such that 
\begin{equation}\label{May_24_1}
\<\( 2\dd\( \(Cv\)\circ d\)-C\)H,H\><0.
\end{equation}
From Proposition~\ref{pgct}, we have that $h:=[0,H]\in \T_{\Ka}\(R\)$ and $w:=\left[ -2\dd\(HH^\top\)\circ d,0\right]\in \T^2_{\Ka}(R,h).$ Define
\begin{equation}\label{May_24_2}
\widehat{R}\(t,H\)=\left[ v-t^2\big(\dd(HH^\top)\circ d\big),tH\right]\in \R^{n\times r}.
\end{equation}
We have that $\dist\big(\widehat{R}\(t,H\),\Ka\big)=o(t^2).$ Define the function $F:\R^{n\times r}\rightarrow \R$ such that $F(R)=-\<C,RR^\top\>$ is the objective function of (SQKELR). We have that
\begin{align}\label{May_24_3}
&F\(\P_{\Ka}\big(\widehat{R}\(t,H\)\big)\)=F\(\widehat{R}\(t,H\)\)+o(t^2)\notag \\[2pt]
&=F(R)+ t^2 \inprod{\( 2\dd\( \(Cv\)\circ d\)-C\)H}{H} +o(t^2),
\end{align}
where the second equality is obtained by expanding $\widehat{R}\(t,H\)\widehat{R}\(t,H\)^\top$ and some algebraic manipulations.
From (\ref{May_24_1}) and (\ref{May_24_3}), we have that for $t>0$ sufficiently small, 
\begin{equation}
F\(\P_{\Ka}\(\widehat{R}\(t,H\)\)\)<F(R).
\end{equation}
 This contradicts to the assumption that $R$ is a local minimum of (SQKELR).
 Hence the lemma is proved.
\end{proof}

\medskip
Lemma~\ref{esc} provides us a necessary condition for a non-regular point of (SQKELR) to be a local minimum. The equality (\ref{May_24_3}) also provides us a way to escape from any non-regular point that does not satisfy the necessary condition. Next, we give the following definition.
\begin{defi}\label{escdirec}
Suppose $(a,\tau)$ satisfies Assumption~\ref{ass1} and $r\geq 2.$  Let $R=ve_1^\top\in \N_{n,r},$ where $v\in \{0,1\}^n$ and $d\in \{-1,1\}^n,\sigma_v\in \{-1,1\}$ are defined as in Proposition~\ref{pgct}. We say that $H\in \R^{n\times (r-1)}$ is an escaping direction of $R$ if $H$ is a feasible solution of (\ref{escpro}) with a negative function value.
\end{defi}
Obviously, we can escape from a non-regular point of (SQKELR) by moving along an escaping direction of it. However, if there is no escaping direction, what can we tell about 
a non-regular point $R=ve_1^\top\in \N_{n,r}$? To answer the question, we first consider the following SDP relaxation of problem (\ref{escpro})
\begin{equation}\label{escSDP}
\min\left\{ \< 2\dd\( \(Cv\)\circ d\)-C,X\>:\  
\begin{array}{l}
\<aa^\top-\sigma_v\tau\cdot\dd\(a\circ d\),X\>=0,\\[5pt]
 \Tr(X)=1,\ X\in \S^n_+ 
\end{array} 
\right\}.
\end{equation}
The dual problem of (\ref{escSDP}) is as follows:
\begin{equation}\label{escSDPd}
\max\left\{ \beta:\  2\dd\( \(Cv\)\circ d\)-C-\alpha\( aa^\top-\sigma_v\tau\cdot\dd\(a\circ d\)\)-\beta I\in \S^n_+,\alpha,\beta\in \R \right\}.
\end{equation}
We show 
in Lemma~\ref{exikkt} that the duality gap between (\ref{escSDP}) and (\ref{escSDPd}) is zero and the optimal values of both (\ref{escSDP}) and (\ref{escSDPd}) are attainable, i.e., there exists a KKT solution of (\ref{escSDP}).
 
 Now, we can define the second order stationarity of a non-regular point as follows.

\begin{defi}\label{secsta}
Suppose $(a,\tau)$ satisfies Assumption~\ref{ass1} and $r\geq 3.$  Let $R=ve_1^\top\in \N_{n,r},$ where $v\in \{0,1\}^n$ and $d\in \{-1,1\}^n,\sigma_v\in \{-1,1\}$ are defined as in Proposition~\ref{pgct}. We say that $R$ is a second order stationary point of (SQKELR) if the optimal function value of the SDP (\ref{escSDP}) is nonnegative.
\end{defi}

\medskip
\begin{rem}\label{escrem}
Verifying the second order stationarity of a non-regular point of (SQKELR) is equivalent to solving the SDP problem (\ref{escSDP}), which can be done in polynomial time. Since there are only two affine constraints in (\ref{escSDP}), there exists an optimal solution $X^*$ such that $\rr\(X^*\)\leq 2.$ 
Observe that if $\inprod{\( 2\dd\( \(Cv\)\circ d\)-C\)}{X^*}<0$ and $X^*=HH^\top,$ where $H\in \R^{n\times (r-1)},$ then from Definition~\ref{escdirec}, we have that $H$ is an escaping direction and hence $R=ve_1^\top$ is not a local minimum for (SQKELR).
\end{rem}

With Definition~\ref{secsta}, we have the following simple corollary.

\begin{coro}\label{cosecsta}
Suppose $a,\tau$ satisfy Assumption~\ref{ass1} and $r\geq 3.$ Let $R=ve_1^\top\in \N_{n,r},$ where $v\in \{0,1\}^n$ and $d\in \{-1,1\}^n,\sigma_v\in \{-1,1\}$ are defined as in Proposition~\ref{pgct}. If $R$ is a local minimum of (SQKELR), then $R$ is a second order stationary point.
\end{coro}

\subsection{When is a second order stationary point of (SQKELR) global optimal.}\label{Global}
In this subsection, with the definition of second order stationarity, we move on to study when is a second order stationary point of (SQKELR) also a global optimal solution. The following lemma gives a sufficient condition for a second order stationary point of (SQKELR) to be a global optimal solution. 

%%%%%%%%%%%%%%%%%%%%%%%%%%%%%%%%rank deficient%%%%%%%%%%%%%%%%%%%%%
\begin{lem}\label{rankdefi}
Suppose $(a,\tau)$ satisfies Assumption~\ref{ass1} and $r\geq 3.$ If $R\in \Ka$ is a second order stationary point of (SQKELR) such that $\rr\(R\)\leq r-2$, then 
$Y:=\begin{pmatrix}e_1&R^\top\end{pmatrix}^\top\begin{pmatrix}e_1&R^\top\end{pmatrix}$ is a global optimal solution of (SQKE) and so $R$ is a global optimal solution of (SQKELR).
\end{lem}

\begin{proof}
\pf
\noindent{\bf Case 1.} $R$ is regular.\\
Since $R\in \Ka$ is a regular point, the LICQ holds at $R.$ Thus, it satisfies the following first order and second order KKT conditions for (SQKELR), i.e., there exists $\(\mu,\lambda\)\in \R^n\times \R$ such that 
\begin{align}
&\|a^\top R\|^2-\tau a^\top Re_1=0,\ \dd\big(RR^\top\big)-Re_1=0,\label{KKTpri}
\\[2pt]
&-2CR-2\dd(\mu)R+\mu e_1^\top-2\lambda aa^\top R+\lambda \tau ae_1^\top=0, \label{KKTdua}\\[2pt]
& \inprod{-2C-2\dd\(\mu\)-2\lambda aa^\top}{HH^\top} \geq 0\ \forall H\in \T_{\Ka}\(R\).\label{KKTsec}
\end{align}
Because LICQ is satisfied at $R,$ $\T_{\Ka}\(R\)$ corresponds to the following tangent space of $\Ka$ at $R$
\begin{equation}\label{May_24_7}
T_R\Ka:=\left\{h\in \R^{n\times r}:\ 2\inprod{aa^\top R}{H}-\tau a^\top He_1=0,\ 
2\dd\big(RH^\top\big)-He_1=0\right\}.
\end{equation}
Since $\rr\(R\)\leq r-2,$ there exists $b\in \R^r$ such that $b\neq 0,$ $b_1=0$ and $Rb=0.$ It is easy to see that for any $d\in \R^n,$ $db^\top\in T_R\Ka.$ Substituting $db^\top$ into (\ref{KKTsec}), we get that 
\begin{equation}\label{KKTsecs}
-2C-2\dd\(\mu\)-2\lambda aa^\top\in \S^n_+.
\end{equation}
Now, we will move on to use (\ref{KKTpri}), (\ref{KKTdua}) and (\ref{KKTsecs}) to prove the global optimality of $R.$ Note that we only have to prove that $Y=\(e_1\ R^\top\)^\top \(e_1\ R^\top\)$ is a global optimal solution of (SQKE) through the KKT conditions of (SQKE). First rearrange (\ref{KKTdua}) to get the following equation
\begin{equation}\label{May_24_8}
\( \frac{\mu+\lambda \tau a}{2}, -C-\dd\(\mu\)-\lambda aa^\top \) \begin{pmatrix} e_1^\top \\ R\end{pmatrix}=0.
\end{equation}
After multiplying $R^\top$ to the left-hand side and $e_1$ to the right-hand side of (\ref{KKTdua}), we get the following equation
\begin{equation}\label{May_24_9}
R^\top \( -C-\dd\(\mu\)-\lambda aa^\top\)Re_1+\frac{1}{2}R^\top \mu +\frac{1}{2}\lambda \tau R^\top a=0.
\end{equation}
Dividing equation (\ref{KKTdua}) by 2 and taking its transpose, we get the following equation
\begin{equation}\label{May_24_10}
R^\top\(-C-\dd\(\mu\)-\lambda aa^\top\)=-\frac{e_1\mu^\top+\lambda \tau e_1a^\top}{2}.
\end{equation}
Substituting (\ref{May_24_10}) into (\ref{May_24_9}) and taking its transpose, we get

\begin{equation}\label{May_24_12}
\frac{-\(\mu^\top+\lambda \tau a^\top\)Re_1}{2}e_1^\top+\frac{\mu^\top+\lambda \tau a^\top}{2}R=0.
\end{equation}
Combining (\ref{May_24_8}) and (\ref{May_24_12}), we get
\begin{equation}\label{May_24_12.5}
\begin{pmatrix} \frac{-\(\mu^\top+\lambda \tau a^\top\)Re_1}{2} & \frac{\mu^\top+\lambda \tau a^\top}{2}\\[5pt] \frac{\mu+\lambda \tau a}{2} & -C-\dd\(\mu\)-\lambda aa^\top\end{pmatrix}\begin{pmatrix} e_1^\top \\R\end{pmatrix}=0.
\end{equation}
Note that matrix $Y:=\begin{pmatrix} e_1&R^\top\end{pmatrix}^\top
\begin{pmatrix} e_1&R^\top\end{pmatrix}$ is inside the feasible region of (SQKE). Now consider the problem (SQKE). If we consider 
%%$\begin{pmatrix}0&0\\0&-C\end{pmatrix}$ 
$(0,0\,;\,0,-C)$ 
as its {profit} matrix, $\mu$ as the dual variable of the
constraint $\dd\(X\)=x,$ $\lambda$ as the dual variable of the constraint $a^\top Xa-\tau a^\top x=0$ and $\frac{1}{2}\(\mu^\top+\lambda \tau a^\top\)Re_1$ as the dual variable for $Y_{11}=1,$ then (\ref{May_24_13}) implies that the complementarity condition of (SQKE) hold for the primal variable $Y$ and the dual variables just mentioned. In order to prove that $Y$ is an optimal solution of (SQKE), we only have to check its dual feasibility i.e., whether the dual slack ${\bf S}$ of (SQKE) is positive semidefinite, where
\begin{equation}\label{dslack}
{\bf S}:=\begin{pmatrix} \frac{-\(\mu^\top+\lambda \tau a^\top\)Re_1}{2} & \frac{\mu^\top+\lambda \tau a^\top}{2}\\[5pt] \frac{\mu+\lambda \tau a}{2} & -C-\dd\(\mu\)-\lambda aa^\top\end{pmatrix}.
\end{equation}
The positive semidefiniteness of ${\bf S}$ comes from (\ref{KKTsecs}) and the following equation
\begin{equation}
{\bf S}=\begin{pmatrix} -\(Re_1\)^\top \\ I\end{pmatrix}\(-C-\dd\(\mu\)-\lambda aa^\top\)\begin{pmatrix} -Re_1 & I\end{pmatrix},
\end{equation}
which can be verified by using (\ref{May_24_10}) several times.\\
\noindent{\bf Case 2.} $R$ is nonregular.\\
In this case, suppose $R=ve_1^\top\in \N_{n,r},$ where $v\in \{0,1\}^n$ and $d\in \{-1,1\}^n,\sigma_v\in \{-1,1\}$ are defined as in Proposition~\ref{pgct}. From Definition~\ref{secsta}, we have that the optimal value of (\ref{escSDP}) is nonnegative. From Lemma~\ref{exikkt}, there exists a KKT solution $\(X,\alpha,\gamma\)$ of (\ref{escSDP}) that satisfies the following KKT condition
\begin{align}
&\<aa^\top-\sigma_v\tau\cdot\dd\(a\circ d\),X\>=0,\ \Tr\(X\)=1,\ X\in \S^n_+\label{May_24_13}\\[2pt]
&2\dd\( \(Cv\)\circ d\)-C-\alpha \( aa^\top-\sigma_v\tau\cdot\dd\(a\circ d\)\)-\gamma I\in \S^n_+\label{May_24_14}\\[2pt]
&\< 2\dd\( \(Cv\)\circ d\)-C-\alpha \( aa^\top-\sigma_v\tau\cdot\dd\(a\circ d\)\)-\gamma I,X\>=0.\label{May_24_15}
\end{align}
Also, since the optimal function value of (\ref{escSDP}) is nonnegative, we have that 
\begin{equation}\label{May_24_16}
\<2\dd\( \(Cv\)\circ d\)-C,X\>\geq 0.
\end{equation}
Now, combining (\ref{May_24_13}), (\ref{May_24_15}) and (\ref{May_24_16}), we have that $\gamma\geq 0.$ Then, from (\ref{May_24_14}), we have that 
\begin{equation}\label{May_24_17}
2\dd\( \(Cv\)\circ d\)-C-\alpha \( aa^\top-\sigma_v\tau \cdot\dd\(a\circ d\)\)\in \S^n_+.
\end{equation}
Comparing (\ref{May_24_17}) and (\ref{KKTsecs}), we may choose the dual variables of (SQKELR) as follows
\begin{equation}\label{May_24_18}
\mu:=-2\(Cv\)\circ d-\alpha \sigma_v\tau \cdot a\circ d,\quad \lambda:=\alpha.
\end{equation}
Then from (\ref{May_24_17}),\ (\ref{KKTsecs}) holds. Now, in order to prove $Y:=\begin{pmatrix}e_1& R^\top\end{pmatrix}^\top
\begin{pmatrix}e_1& R^\top\end{pmatrix}$ 
is the global optimal solution of (SQKE), we only have to verify whether (\ref{KKTdua}) holds for $R,$ $\mu$ and $\lambda.$ Because once (\ref{KKTdua}) holds, the remaining proof will be exactly the same as in Case 1. We have that
\begin{align}
&-2CR-2\dd\(\mu\)R+\mu e_1^\top-2\lambda aa^\top R+\lambda \tau ae_1^\top\notag \\[2pt]
&=\( -2Cv-2\mu\circ v+\mu-2\lambda aa^\top v+\lambda \tau a\)e_1^\top\notag\\[2pt]
&=\( -2Cv-2\(-2\(Cv\)\circ d-\alpha \sigma_v\tau\cdot a\circ d\)\circ v-2\(Cv\)\circ d-\alpha \sigma_v\tau\cdot a\circ d-2\alpha aa^\top v+\alpha \tau a\)e_1^\top\notag \\
&=\Big( -2Cv+4(Cv)\circ v-2\(Cv\)\circ d\Big)e_1^\top
+\tau\alpha\Big( 2\sigma_v a\circ v-\sigma_v a\circ d-\frac{2a^\top v}{\tau}a+a\Big)e_1^\top\notag \\
&=\tau\alpha\Big( \sigma_v a-\frac{2a^\top v}{\tau}a+a\Big)e_1^\top,
\label{May_24_19}
\end{align}
where the last equality comes from that $d=2v-e.$ If 
$v\not=0$, then $\sigma_v=1,$ together with the fact that  $v^\top a=\tau$, we get 
$\sigma_v a-\frac{2a^\top v}{\tau}a+a=0.$ Also, if $v=0$, then $\sigma_v=-1,$  and $\sigma_v a-\frac{2a^\top v}{\tau}a+a=0.$ Thus, we have verified (\ref{KKTdua}).
\end{proof}

\medskip
\begin{rem}\label{recoverdual}
Compared with the results in \cite{Boumal1,Boumal2}, our rank deficient condition is slightly stronger, i.e., there exists at least two zero singular values of a certain second order stationary point. One reason is that the first row of the factorization of $Y$ is fixed to be $e_1^\top$. Also, note that the proof of Lemma~\ref{rankdefi} provides a way to recover the dual variable of (SQKELR), which is also the dual variable for (SQKE). When $R$ is regular, {we} can solve (\ref{KKTdua}) to get a least square solution $(\mu,\lambda).$ The corresponding LICQ property implies that there exists a unique least square solution of (\ref{KKTdua}), which is exactly the dual variable. When $R$ is non-regular, then (\ref{May_24_18}) is an explicit formula for the dual variable of (SQKELR).
\end{rem}

Now, we move on to study when will the rank deficient condition in Lemma~\ref{rankdefi}  hold. The following result is inspired by Lemma 3 in \cite{Boumal0} and Theorem 2 in \cite{Boumal1}. However, our rank-bound is also related to the minimum rank number of the underlying graph of $C$ in (SQKELR).
\begin{lem}\label{generic}
Suppose $(a,\tau)$ satisfies Assumption~\ref{ass1}, $r\geq 3$ and $R$ is a regular first order stationary point of (SQKELR). The following statements hold.
\begin{itemize}
\item [(i)] If $\frac{(r-2)(r-1)}{2}>n+1,$ then for a generic $C\in \S^n,$ $\rr\(R\)\leq r-2.$
\item [(ii)] If $C\in \Q(G)$ for some graph $G$ and $r\geq n+4-\mr(G),$ then $\rr\(R\)\leq r-2.$
\end{itemize}
\end{lem}
\begin{proof}
\pf
Because $R$ is a regular first order stationary point of (SQKELR), we have that the KKT condition (\ref{KKTdua}) and its variant (\ref{May_24_8}) hold at $R.$ After applying Sylvester's rank inequality (see 10.5 of \cite{matrixcook}) to (\ref{May_24_8}), we have that
\begin{equation}\label{May_24_20}
\rr\( -C-\dd\(\mu\)-\lambda aa^\top\)\leq n+1-\rr\(R\).
\end{equation}
Now, suppose $\rr\(R\)\geq r-1$ and $\frac{(r-2)(r-1)}{2}>n+1,$ then from (\ref{May_24_20}), we have that
\begin{equation}\label{May_24_21}
\rr\( -C-\dd\(\mu\)-\lambda aa^\top\)\leq n+2-r,
\end{equation}
which implies that $-C-\dd\(\mu\)-\lambda aa^\top\in \M_{\leq n+2-r},$ where 
$$\M_{\leq n+2-r}:=\left\{ A\in \S^n:\ \rr\(A\)\leq n+2-r\right\}.$$
Then we have that $-C\in \M_{\leq n+2-r}+\left\{ \dd\(\mu\)+\lambda aa^\top:\ \mu\in \R^n,\ \lambda\in \R\right\}$ whose dimension is bounded by $\frac{n(n+1)}{2}-\frac{(r-2)(r-1)}{2}+n+1<\frac{n(n+1)}{2}=\dim\(\S^n\).$ This means that $C\in \S^n$ is contained in a set of measure zero. Thus, (i) is proved.

If $r\geq n+4-\mr(G)$, then from (\ref{May_24_20}), we have that
\begin{align}
&\rr\(R\)\leq n+1-\rr\( -C-\dd\(\mu\)-\lambda aa^\top\)\notag \\
&\leq n+2-\rr\(-C-\dd\(\mu\)\)\leq n+2-\mr(G)\leq r-2.
\end{align}
Thus, (ii) is verified.
\end{proof}

\medskip
Note that we did not consider non-regular points of $\Ka$ in Lemma~\ref{generic}. This is because from Proposition~\ref{regular}, the rank of any non-regular point $R$ of $\Ka$ already satisfies the condition that $\rr\(R\)\leq 1\leq r-2$ for any $r\geq 3.$ {Now, we are able to state our main result.}
\begin{theo}\label{main}
Suppose $(a,\tau)$ satisfies Assumption~\ref{ass1} and $r\geq 3.$ If $R$ is a second order stationary point of (SQKELR) and at least one of the following two conditions are satisfied
\begin{itemize}
\item [1.] $\frac{(r-2)(r-1)}{2}>n+1$ and $C\in \S^n$ is generic;
\item [2.] $C\in \Q(G)$ for some graph $G$ and $r\geq n+4-\mr(G),$
\end{itemize}
then $Y:=\begin{pmatrix}e_1& R^\top\end{pmatrix}^\top
\begin{pmatrix}e_1& R^\top\end{pmatrix}$ is a global optimal solution of (SQKE) and so $R$ is a global optimal solution of (SQKELR).
\end{theo}

\begin{proof}
\pf
{Theorem~\ref{main} follows directly from the combination of Lemma~\ref{rankdefi} and Lemma~\ref{generic}.}
\end{proof}

%%%%%%%%%%%%%%%%%%%%%%%%%%%Implementation%%%%%%%%%%%%%%%%%%%%%%%%%%
\section{Algorithms and implementation details.}\label{Algorithm}
In this section, we consider how to design a feasible 
algorithm to solve (SQKELR) by using the results established in the last few sections. We discuss several techniques involved in our algorithm.
%%%%%%%%%%%%%%%%%%%%%%%%%%%projection and retraction%%%%%%%%%%%%%%%%%%%%%%
\subsection{Projection and retraction.}\label{PaR}
As mentioned in Corollary~\ref{regular1}, $\Ka$ is a smooth manifold 
%with probability 1 
for a generic input data $(a,\tau)$,
except for a trivial point. Also, even if $a\in \W_{n,\tau},$ which is defined in (\ref{nreg-W}), from (\ref{nreg}), there are only finitely many non-regular points in $\Ka.$ Therefore, we can use a 
Riemannian optimization method to solve (SQKELR) (see \cite{intromani}). Two important operations in a Riemannian optimization method are projection and retraction. Consider a regular point $R\in \Ka.$ Because LICQ holds at $R,$ $\Ka$ is locally a smooth manifold around $R.$ The projection mapping of $\Ka$ at $R$ is simply the 
orthogonal projection operator of  its tangent space $T_R \Ka.$ However, the retraction mapping (see definition 3.41 and definition 5.41 in \cite{intromani} for the definition of retraction and second order retraction) $\Re_R:\T_R \Ka\rightarrow \Ka$ is not unique. For simple manifolds like the unit sphere $\s^{r-1}$, its retraction can be defined as a normalisation mapping $x\rightarrow \frac{x}{\|x\|_2}$. This retraction has an explicit formula and is also the closest point retraction. For more complicated manifold, it may be difficult for us to find a retraction mapping with a simple closed form. In this case, we have to compute the retraction iteratively (see \cite{PR,NR}). In our algorithm, we use the Newton retraction introduced in \cite{NR}. This retraction is equivalent to applying the Gauss-Newton method to the nonlinear system of the manifold and it is initialized by the point at which we want to retract (see algorithm 1 of \cite{NR}). It has been proved that this is a second order retraction in \cite{NR}.

%%%%%%%%%%%%%%%%%%%%%%%%%%%%Escaping%%%%%%%%%%%%%%%%%%%%%%%%%%%%
\subsection{A Riemannian optimization approach that can escape from non-optimal singular points.}\label{RieRound}
In this subsection, we suppose $(a,\tau)$ satisfies Assumption~\ref{ass1} and $r\geq 3.$ When we meet a non-regular point $R=ve_1^\top$ of $\Ka,$ we may first solve (\ref{escSDP}). If the optimal value of (\ref{escSDP}) is negative, then from Remark~\ref{escrem}, we can factorize the optimal solution to get an escaping direction. After that we can move along this direction to escape from this non-optimal non-regular point while reducing the function value of (SQKELR). If the optimal value of (\ref{escSDP}) is nonnegative, i.e., $R$ is a non-regular second order stationary point, then from Lemma~\ref{rankdefi}, we have that 
$Y:=\begin{pmatrix}e_1&R^\top\end{pmatrix}^\top\begin{pmatrix}e_1&R^\top\end{pmatrix}$ is an optimal solution of (SQKE) and so $R$ is a global optimal solution of (SQKELR). 
In practice, it may happen that the iterations of a Riemannian optimization method 
slowly approach a non-regular point of $\Ka$. Thus, to mitigate the slow convergence towards a non-regular point,
we have to apply rounding to every nearly non-regular point to a non-regular point. Due to the constraint $\dd\(RR^\top\)=Re_1,$ we have that $Re_1\in [0,1]^n.$ Define the rounding function $\Ron:\Ka\rightarrow \{0,1\}^n e_1^\top$ such that for any $R\in \Ka,$ $\Ron(R)=ve_1^\top$ and
\begin{equation}\label{round}
v_i=\begin{cases} 1& {\rm if}\ \(Re_1\)_i\geq 0.5 \\ 0 & {\rm otherwise}\end{cases}.
\end{equation}
With the function $\Ron(\cdot),$ we may define the following set for any $\delta>0.$
\begin{equation}\label{Kadelta}
\Ka\(\delta\):=\left\{ R\in \Ka:\ \|R-\Ron(R)\|<\delta\right\}.
\end{equation}
It is easy to see that $\N_{n,r}\in \Ka(\delta)$ for any $\delta>0.$ 
\begin{lem}\label{roundlem}
Suppose $(a,\tau)$ satisfies Assumption~\ref{ass1} and $r\geq 3.$ Then there exists $\delta_0>0$ such that for any $\delta\in (0,\delta_0)$ and any $R\in \Ka\(\delta\),$ $\Ron(R)\in \N_{n,r}$.
\end{lem}
\begin{proof}
\pf
Assume on the contrary that there exists $\delta_k\downarrow 0$ such that for any $k\in \mathbb{N}^+,$ there exists $R_k\in \Ka\(\delta_k\)$ such that $\Ron\(R_k\)=v_k e_1^\top\notin \N_{n,r}.$ This implies that 
\begin{equation}\label{May_24_22}
\|a^\top v_k\|^2-\tau a^\top v_k\neq 0,\ \forall k\in \mathbb{N}^+.
\end{equation}
Since there are only finitely many elements in $\{ \|a^\top v\|^2-\tau a^\top v:\ v\in \{0,1\}^n\},$ we have that there exists $\epsilon>0$ such that 
\begin{equation}\label{May_24_23}
\left|\|a^\top v_k\|^2-\tau a^\top v_k\right|\geq \epsilon,\ \forall k\in \mathbb{N}^+.
\end{equation}
However, from $R_k\in \Ka(\delta_k),$ we have that $\|a^\top R_k\|^2-\tau a^\top R_ke_1=0$ and $\lim\limits_{k\rightarrow \infty}\|R_k-v_ke_1^\top\|=0.$ Thus, we have that
\begin{align}\label{May_24_24}
&\lim\sup\limits_{k\rightarrow \infty} \left|\|a^\top v_k\|^2-\tau a^\top v_k\right|=\lim\sup\limits_{k\rightarrow \infty} \left|\|a^\top v_ke_1^\top\|^2-\tau a^\top v_ke_1^\top e_1\right|\notag \\
&=\lim\sup\limits_{k\rightarrow \infty}\left|\|a^\top R_k\|^2-\tau a^\top R_k e_1\right|=0,
\end{align}
which contradicts to (\ref{May_24_23}). 
\end{proof}
With Lemma~\ref{roundlem}, we can design the following algorithm, which can be seen as a modified version of algorithm 1 in \cite{GEP}.

\begin{description}
\item [\bf Algorithm 1:] Define $f\(R\):=\<-C,RR^\top\>.$ Choose $R_0\in \Ka$ and $\delta_0\in (0,1).$ Set $k=0.$

\item[Step 1.] Set $R_k$ as the initial point, $f_k:=f(R_k).$ Use any Riemannian optimization method to solve (SQKELR) such that all the generated iteration point $R^{(j)} \in \Ka$
({$\forall j\in \{0,1,2,\ldots\}$}, with 
$R^{(0)} = R_k$) satisfies $f(R^{(j)})\leq f_k$. 
If the KKT residue
corresponding to $R^{(j)}$ as defined in 
\eqref{KKTres} for the convex SDP (SQKE)
 satisfies the required accuracy tolerance, terminate the Riemannian optimization method and stop. If at some iteration of the Riemannian optimization method, 
 we encounter an iterate 
 denoted by $\widehat{R}_k$ that satisfies $\widehat{R}_k\in \Ka(\delta_k),$ goto Step 2.
\item[Step 2.]
{\bf Case 1.} $\Ron\big(\widehat{R}_k\big)\notin \Ka$: set $R=\widehat{R}_k, $ goto step 2'.\\
{\bf Case 2.} $\Ron\big(\widehat{R}_k\big)\in \Ka$ and it is an optimal solution: Stop.\\
{\bf Case 3.} $\Ron\big(\widehat{R}_k\big)\in \Ka$ and is non-optimal: Use the escaping strategy described above to find another regular point {(regularity comes from the finiteness of non-regular points of $\Ka$)} $R^+\in \Ka$ such that $f(R^+)<f(\Ron(\widehat{R}_k))-\epsilon\big( \Ron\big(\widehat{R}_k\big)\big),$ where 
$\epsilon\big( \Ron\big(\widehat{R}_k\big)\big)>0$ is only related to 
$\Ron\big(\widehat{R}_k\big).$ If $f(R^+)<f(\widehat{R}_k)$ set $R\leftarrow R^+;$ Otherwise, set $R\leftarrow \widehat{R}_k.$ {(In the latter case, $R$ is also a regular point. This is because otherwise $f(R^+)<f(\Ron(\widehat{R}_k))=f(\widehat{R}_k)$ and this case won't happen.)}  Goto Step 2'.
\item[Step 2'.] Set $\delta_{k+1}\leftarrow \delta_k/2$, $R_{k+1}\leftarrow R$,\ $k\leftarrow k+1,$ goto Step 1.
\end{description}

\medskip
\begin{rem}\label{finitecon}
Similar to Theorem 4.1 of \cite{GEP}, {we} can show that for any initialization, Algorithm 1 will terminate after finitely many outer iterations. Here we omit the proof because of their similarity. In other words, after finitely many escaping steps, the iterations of Algorithm 1 will either terminate at some optimal non-regular point or enter a subset $\Ka\setminus \Ka\(\delta_k\)$ of $\Ka$ that is bounded away from the non-regular points of $\Ka.$ In this case, Algorithm 1 is reduced to optimization on a smooth manifold and any results in Riemannian optimization can be applied in our algorithm.
\end{rem}
In Step 1 of Algorithm 1, we choose a Riemannian gradient descent method with Barzilai-Borwein (BB) step and non-monotone linesearch (see \cite{BB,RieBBrank,RieBB,BB1}). We omit the details of the algorithm because it is frequently used in the literature. The parameters we choose for BB step is the same as the Riemannian BB algorithm in \cite{GEP}. 

\subsection{A rounding procedure to the output of (SQKELR).}\label{Rounding}
In practice, {we} can apply rounding procedures to the output of the SDP relaxations of integer programming problems to get an approximate solution. A famous example is the SDP relaxation of a max-cut problem. In \cite{maxcut}, Goemans and Williamson designed a randomized rounding procedure for the output of a max-cut SDP whose expected value is at least 0.878 of the optimal value of the max-cut problem. In this subsection, we describe a simple rounding procedure for the output of the problem (SQKELR). Suppose $R\in \R^{n\times r}$ is the optimal solution of (SQKELR) and $d=(d_1,d_2,\ldots,d_n)^\top$ is the first column of $R.$ We first sort the entries of $d$ in descending order i.e., $d_{i_1}\geq d_{i_2}\geq \ldots \geq d_{i_n}.$ Let $k\in [n]$ to be the maximum integer such that $\sum_{\ell=1}^k a_{i_\ell}\leq \tau.$ Then, define the output $x\in \{0,1\}^n$   of our rounding procedure as follows: $x_{i_\ell}=1$ if $\ell\leq k$ and $x_{i_\ell}=0$ if $\ell>k.$ It is easy to see that $x$ is a feasible solution for (QKP). Also, the complexity of our rounding procedure is $O(n\log n)$ because of the sorting step. The (SQK) model together with the rounding procedure can be considered as a heuristic to find a lower bound of (QKP). We can measure the accuracy of the lower bound by comparing it with the upper bound provided by the SDP relaxation (SQK). {(Here We remove the minus sign of the objective function of (SQK) to get the upper bound of (QKP).)} We will test the accuracy and efficiency of this method later in section~\ref{IntQKP}.

%%%%%%%%%%%%%%%%%%%%%%%%%%%%%%Numerical experiment%%%%%%%%%%%%%%%%%%%%%%
\section{Numerical experiments.}\label{Numerical}
In this section, we conduct numerical experiments to verify the efficiency of our feasible method described in the previous section.  Our dataset satisfies Assumption~\ref{ass1} and the profit matrix $C$ is always nonnegative. Then, from Proposition~\ref{equiv}, (SQK) is equivalent to (SQKE) so we directly test the problem (SQKE). We use the KKT residue (given by 
$\max\{{\rm Rp},{\rm Rd},{\rm pdgap}\}$) for (SQKE) to measure the accuracy of a solution $Y=\(\alpha, x^\top;x, X\)\in \S^{n+1}_+$, where
\begin{align}\label{KKTres}
& {\rm primal\ residue:}\ {\rm Rp}:=\frac{\sqrt{\|\dd\(X\)-x\|^2+(\alpha-1)^2+\(a^\top Xa-\tau a^\top x\)^2}}{2},\notag\\
& {\rm dual\ residue:}\ {\rm Rd}:=\frac{\| \Pi_{\S^{n+1}_-}\({\bf S}\)\|}{1+\|{\bf S}\|},\ {\rm duality\ gap:}\ {\rm pdgap:}=\frac{|\<-C,X\>-y|}{1+|\<-C,X\>|+|y|},
\end{align}
where ${\bf S}$ is the dual slack of (SQKE) defined in (\ref{dslack}), and the dual variable $\(\mu,\lambda\)$ can be constructed from Remark~\ref{recoverdual}; $y$ is the dual variable for the constraint $Y_{11}=1.$ For our method, $y=\frac{1}{2}\(\mu^\top+\lambda \tau a^\top\)Re_1.$ {Note that Algorithm 1 uses factorization so we {do not} have to check whether $Y$ is positive semidefinite.}

The algorithms we compare are SDPLR by Burer and Monteiro and SketchyCGAL by Yurtsever et al. in \cite{CGAL}. {Their source codes can be downloaded from \href{https://sburer.github.io/projects.html}{https://sburer.github.io/projects.html} and \href{https://github.com/alpyurtsever/SketchyCGAL}{https://github.com/alpyurtsever/SketchyCGAL}, respectively.} Note that all  three algorithms make use of the low rank property of the linear SDP to save storage. Therefore, our comparison will demonstrate the advantage of exploiting the geometric property of a problem. We will stop the Riemannian optimization method in Step 1 of Algorithm 1 when the 
KKT residue $\max\{ {\rm Rp},\ {\rm Rd},\ {\rm pdgap}\}<10^{-6}.$ Because it is expensive to compute the {\rm Rd} too frequently, we only check the KKT residue when the normalized norm of Riemannian gradient is small enough, i.e.,
\begin{equation}\label{riegrad}
\frac{\|\P_{T_R\Ka} \(-2CR\)\|}{\max\{1,\|R\|\}}<{\rm tolg}.
\end{equation} 
We first set tolg $:=10^{-6}.$ When (\ref{riegrad}) is satisfied but $\max\{ {\rm Rp},\ {\rm Rd},\ {\rm pdgap}\}\geq 10^{-6},$ we will update the tolerance as ${\rm tolg}\leftarrow {\rm tolg}/10$ and continue running our algorithm. Note that we will always check the KKT residue for a non-regular point because its Riemannian gradient may not be well-defined.

We also set the stopping tolerance for SDPLR to be $10^{-6}.$ For SketchyCGAL, we find that it cannot find an accurate solution, so we set the tolerance to be $10^{-1}.$ The parameter $\delta_0$ in Algorithm 1 is chosen to be $0.1.$ For Algorithm 1, we use random initialization with normal distribution. For SDPLR and SketchyCGAL, we use their defaulted initialization. From our experiment, SDPLR and SketchCGAL usually fail to converge without a pre-scaled input data. Therefore, in order to improve their performance, we scale the capacity $\tau$ and weight vector $a$ as $a\leftarrow a/\tau$ and $\tau\leftarrow 1.$ For all three algorithms, we set the maximum running time to be 3600 seconds. All the experiments are run using Matlab R2021b on a Workstation with a Intel(R) Xeon(R) CPU E5-2680 v3 @ 2.50GHz Processor and 128GB RAM.

%%%%%%%%%%%%%%%%%%%%%%%%%%KP test$$$$$$$$$$$$$$$$$$$$$$$$$$$$$$$$$$$$$$$$$
\subsection{The SDP relaxation of Binary linear knapsack problem.}\label{BKP}
In this subsection, we consider binary linear knapsack problems. In this case, $C$ is a diagonal matrix. We use the randomly generated data in \cite{hardKP}. {The dataset can be downloaded from \href{http://artemisa.unicauca.edu.co/~johnyortega/instances\_01\_KP/}{http://artemisa.unicauca.edu.co/~johnyortega/instances\_01\_KP/}.} In order to verify the new rank-bound from Theorem~\ref{MGrank}, we choose $r=3$ in Algorithm 1. Note that smaller $r$ does not always lead to faster convergence because it may introduce more spurious local minima. From our numerical tests, we find that the behaviours of SDPLR and SketchyCGAL will always be better if we choose $r$ to be slightly larger than $3.$ Therefore, we choose their rank to be $r=10$ to give them some advantage over Algorithm 1 while the problem size is still $O(n)$.
Note that we have also tested the setting $r=20$, but the performance 
of SDPLR and SketchyCGAL 
are rather similar to that for the setting $r=10$.
%\fbox{\red{Please confirm the last statement: }}
%\fbox{\blue{Yes, when $r=20,$ SDPLR becomes slower and SketchyCGAL behaves almost the same.}}

\begin{center}
\begin{footnotesize}
\begin{longtable}{|c|c|cccc|l|}
\caption{Comparison of Algorithm1, SDPLR and SketchyCGAL for binary linear knapsack SDP. knapPI\_1, knapPI\_2 and knapPI\_3 are uncorrelated instances, weakly correlated instances and strongly correlated instances, respectively (see \cite{hardKP}). }
\label{LKP}
\\
\hline
problem & algorithm & obj & Rp & Rd & pdgap & time \\ \hline 
 \endhead
 knapPI\_1 & Algorithm1 &  -5.4538020e+04  &  1.64e-10 & 5.38e-09  & 1.46e-08 &  2.89e-01 \\
 n = 1000  & SDPLR &  -5.4537907e+04  &  9.83e-07 & 2.48e-06   & 5.59e-05 &  9.47e+00 \\
 opt = 54503 & SketchyCGAL &   -5.1024775e+04  &  6.27e-02 & 1.45e-08   & 4.77e-02 &  3.19e+02 \\ 
 \hline
 knapPI\_2 & Algorithm1 &  -9.0573608e+03  &  6.46e-16 & 9.93e-10  & 5.12e-08 &  5.88e-01 \\
 n = 1000  & SDPLR &  -9.0545910e+03  &  8.19e-07 & 3.24e-05   & 4.05e-04 &  1.05e+01 \\
 opt = 9052 & SketchyCGAL &   -8.5724562e+03  &  9.68e-02 & 0.00e+00   & 5.15e-02 &  2.08e+03 \\ 
 \hline
 knapPI\_3 & Algorithm1 &  -1.4406317e+04  &  2.23e-13 & 1.13e-07  & 7.86e-09 &  3.98e-01 \\
 n = 1000  & SDPLR &  -1.4406040e+04  &  3.37e-07 & 4.73e-05   & 2.72e-04 &  2.00e+01 \\
 opt = 14390 & SketchyCGAL &   -1.4364806e+04  &  1.41e-01 & 0.00e+00   & 4.96e-02 &  5.98e+02 \\ 
 \hline
 knapPI\_1 & Algorithm1 &  -1.1064590e+05  &  6.64e-12 & 1.57e-08  & 2.07e-09 &  1.02e+00 \\
 n = 2000  & SDPLR &  -1.1064600e+05  &  8.91e-07 & 1.27e-06   & 6.86e-05 &  8.81e+01 \\
 opt = 110625 & SketchyCGAL &   -1.0143376e+05  &  1.06e-01 & 6.86e-11   & 4.71e-02 &  7.29e+02 \\ 
 \hline
 knapPI\_2 & Algorithm1 &  -1.8052806e+04  &  2.13e-09 & 6.93e-07  & 3.79e-07 &  6.24e-01 \\
 n = 2000  & SDPLR &  -1.8051935e+04  &  9.85e-07 & 7.43e-05   & 1.11e-05 &  1.16e+02 \\
 opt = 18051 & SketchyCGAL &   -1.7201318e+04  &  1.97e-01 & 0.00e+00   & 2.02e-01 &  3.60e+03 \\ 
 \hline
 knapPI\_3 & Algorithm1 &  -2.9012870e+04  &  9.18e-11 & 4.15e-10  & 5.72e-09 &  1.44e+00 \\
 n = 2000  & SDPLR &  -2.9010766e+04  &  9.19e-07 & 8.41e-05   & 9.98e-05 &  1.36e+02 \\
 opt = 28919 & SketchyCGAL &   -2.9631203e+04  &  1.77e-01 & 0.00e+00   & 3.78e-02 &  8.14e+02 \\ 
 \hline
 knapPI\_1 & Algorithm1 &  -2.7645834e+05  &  1.71e-12 & 5.99e-07  & 2.39e-08 &  1.24e+01 \\
 n = 5000  & SDPLR &  -2.7645644e+05  &  9.09e-07 & 2.85e-05   & 4.20e-05 &  1.84e+03 \\
 opt = 276457 & SketchyCGAL &   -2.5881374e+05  &  2.28e-01 & 4.30e-08   & 3.35e-02 &  2.33e+03 \\ 
 \hline
 knapPI\_2 & Algorithm1 &  -4.4357542e+04  &  2.45e-12 & 9.60e-08  & 3.05e-08 &  1.24e+01 \\
 n = 5000  & SDPLR &  -4.4354071e+04  &  9.90e-07 & 7.06e-05   & 4.67e-04 &  1.42e+03 \\
 opt = 44356 & SketchyCGAL &   -2.5213507e+04  &  6.57e-01 & 4.97e-08   & 6.07e-01 &  3.60e+03 \\ 
 \hline
 knapPI\_3 & Algorithm1 &  -7.2563369e+04  &  3.51e-09 & 8.99e-10  & 9.00e-09 &  6.67e+00 \\
 n = 5000  & SDPLR &  -7.2518590e+04  &  1.69e-04 & 4.47e-04   & 2.14e-05 &  2.42e+03 \\
 opt = 72505 & SketchyCGAL &   -6.8704702e+04  &  5.81e-01 & 5.52e-08   & 6.02e-02 &  3.60e+03 \\ 
 \hline
 knapPI\_1 & Algorithm1 &  -5.6364672e+05  &  9.87e-10 & 6.70e-07  & 2.30e-07 &  3.13e+01 \\
 n = 10000 & SDPLR &    -  &  - & -   & - &  - \\
 opt = 563647 & SketchyCGAL & -  &  - & -   & - &  - \\
 \hline
 knapPI\_2 & Algorithm1 &  -9.0204226e+04  &  5.55e-11 & 7.45e-07  & 1.23e-07 &  4.53e+01 \\
 n = 10000 & SDPLR &    -  &  - & -   & - &  - \\
 opt = 90204 & SketchyCGAL & -  &  - & -   & - &  - \\
 \hline
 knapPI\_3 & Algorithm1 &  -1.4694938e+05  &  2.50e-09 & 5.12e-08  & 1.61e-09 &  3.95e+01 \\
 n = 10000 & SDPLR &    -  &  - & -   & - &  - \\
 opt = 146919 & SketchyCGAL & -  &  - & -   & - &  - \\
 \hline\end{longtable}
\end{footnotesize}
\end{center}

From the results in Table~\ref{LKP}, we can see that Algorithm 1 can solve all the problems accurately. This verifies the validity of the new rank-bound for binary linear knapsack problems. SDPLR and SketchyCGAL can only solve problems of size up to $5000$ to a moderate
and low accuracy, respectively. For problems of size $10000,$ both of them run out of time and return a solution that is far away from being feasible, so we {do not} show their output in the table. In all these examples, Algorithm 1 is much more efficient than SDPLR and SketchyCGAL. For some instances with $n=5000$, Algorithm 1 is more than 100 times faster than them.

%%%%%%%%%%%%%%%%%%%%%%%QKP test%%%%%%%%%%%%%%%%%%%%%%%%%%%%%%%%%%
\subsection{The SDP relaxation of Quadratic knapsack problem.}\label{NumerQKP}
In this subsection, we consider the quadratic knapsack problem (QKP) with a quadratic profit matrix $C.$ We randomly generate the profit matrix and weight vector following the procedure proposed by Gallo et al. in \cite{QKP}. The entries of the profit matrix $C_{ij}=C_{ji}$ are zero with probability $(1-p)$ and otherwise integers randomly generated uniformly in the range $[1,100].$ We set the problem density as $p=0.25.$ The coefficients of the weight vector $a_i$ are integers uniformly distributed in the range $[1,50].$ While the procedure in \cite{QKP} also randomly generates the knapsack capacity $\tau$ in the range $[1,e^\top a],$ for convenience, we directly set the knapsack capacity to be $\beta\cdot e^\top a,$ where $\beta$ is chosen to be $0.1,0.5$ and $0.9$. This data-generating procedure has been widely accepted in the literature (see \cite{exactQKP,exactlarge,SQK3,knapcom} to just name a few). For Algorithm 1, we choose $r=\lceil\sqrt{2(n+1)}\rceil+2,$ {which satisfies the rank condition in (i)} of Lemma~\ref{generic}. For SDPLR, we use their default rank. For SketchyCGAL, we choose the rank to be $20$ because we find that the rank of the output is always smaller than the classical theoretical upper bound of $\sqrt{2(n+2)}.$ Actually, when we set $r=20$ for Algorithm 1, it can find an accurate optimal solution of (SQKE) even faster than the case for $r=\lceil\sqrt{2(n+1)}\rceil+2$ (see the running time for Algorithm 1 in Table~\ref{QKP} in the next subsection). We choose $r=\lceil\sqrt{2(n+1)}\rceil+2$ for Algorithm 1 to be consistent to Lemma~\ref{generic}. Note that we have also tested the setting $r=50$ for SketchyCGAL, but
its performance is slightly slower than that for the setting $r=20$.
%\fbox{\red{Please confirm the last statement.}}
%\fbox{\blue{Yes, when I choose $r=50$ for sketchyCGAL, the algorithm becomes slightly slower.}}

\begin{center}
\begin{footnotesize}
\begin{longtable}{|c|c|cccc|l|}
\caption{Comparison of Algorithm1, SDPLR and SketchCGAL for quadratic knapsack SDP. }
\label{SQKR}
\\
\hline
problem & algorithm & obj & Rp & Rd& pdgap & time \\ \hline
\endhead
$n$ = 1000 & Algorithm1 & -1.1464123e+06&  2.13e-15 & 4.61e-08  & 1.61e-10 &  1.39e+00 \\ 
 $\beta$ = 1.00e-01 & SDPLR &-1.1464103e+06  &  7.93e-07 &1.03e-07   & 7.44e-06 &  2.37e+01 \\ 
 $p$ = 2.50e-01 & SketchyCGAL &-1.2127791e+06  &  7.41e-02 &1.83e-07   & 1.81e-02 &  7.18e+02 \\ 
 \hline
 $n$ = 1000 & Algorithm1 & -3.6163898e+06&  7.07e-11 & 3.68e-08  & 6.05e-09 &  7.31e-01 \\ 
 $\beta$ = 3.00e-01 & SDPLR &-3.6163884e+06  &  9.96e-07 &1.05e-06   & 2.23e-05 &  5.13e+01 \\ 
 $p$ = 2.50e-01 & SketchyCGAL &-7.3592186e+06  &  3.96e-01 &4.70e-04   & 3.52e-01 &  3.60e+03 \\ 
 \hline
 $n$ = 1000 & Algorithm1 & -6.2329083e+06&  3.05e-15 & 3.47e-08  & 1.01e-09 &  4.96e-01 \\ 
 $\beta$ = 5.00e-01 & SDPLR &-6.2329001e+06  &  9.70e-07 &8.31e-07   & 2.55e-07 &  7.55e+01 \\ 
 $p$ = 2.50e-01 & SketchyCGAL &-9.2273123e+06  &  2.06e-01 &3.34e-02   & 5.88e-01 &  3.60e+03 \\ 
 \hline
 $n$ = 1000 & Algorithm1 & -8.8340915e+06&  2.14e-10 & 2.01e-08  & 1.19e-10 &  6.26e-01 \\ 
 $\beta$ = 7.00e-01 & SDPLR &-8.8341094e+06  &  9.89e-07 &3.92e-07   & 4.59e-06 &  1.17e+02 \\ 
 $p$ = 2.50e-01 & SketchyCGAL &-1.2078637e+07  &  2.11e-01 &4.75e-01   & 9.46e-01 &  3.60e+03 \\ 
 \hline
 $n$ = 2000 & Algorithm1 & -4.8166859e+06&  2.83e-10 & 2.57e-09  & 1.46e-09 &  4.30e+00 \\ 
 $\beta$ = 1.00e-01 & SDPLR &-4.8166801e+06  &  9.53e-07 &6.07e-06   & 2.73e-06 &  1.24e+02 \\ 
 $p$ = 2.50e-01 & SketchyCGAL &- &  -&-  &- & - \\ 
 \hline
 $n$ = 2000 & Algorithm1 & -1.4947423e+07&  4.66e-15 & 4.03e-09  & 1.83e-09 &  2.62e+00 \\ 
 $\beta$ = 3.00e-01 & SDPLR &-1.4947429e+07  &  9.97e-07 &7.55e-07   & 9.12e-06 &  1.14e+03 \\ 
 $p$ = 2.50e-01 & SketchyCGAL &- &  -&-  &- & - \\ 
 \hline
 $n$ = 2000 & Algorithm1 & -2.5092783e+07&  5.67e-10 & 1.03e-08  & 5.42e-11 &  2.53e+00 \\ 
 $\beta$ = 5.00e-01 & SDPLR &-2.5092737e+07  &  8.64e-07 &9.21e-07   & 3.69e-06 &  1.29e+03 \\ 
 $p$ = 2.50e-01 & SketchyCGAL &- &  -&-  &- & - \\ 
 \hline
 $n$ = 2000 & Algorithm1 & -3.5316003e+07&  3.98e-10 & 1.23e-09  & 1.45e-10 &  2.34e+00 \\ 
 $\beta$ = 7.00e-01 & SDPLR &-3.5316024e+07  &  9.76e-07 &1.55e-07   & 1.85e-06 &  1.31e+03 \\ 
 $p$ = 2.50e-01 & SketchyCGAL &-5.0470000e+07  &  3.06e-01 &5.87e-01   & 9.96e-01 &  3.60e+03 \\ 
 \hline
 $n$ = 5000 & Algorithm1 & -3.1439664e+07&  1.02e-11 & 1.07e-09  & 6.85e-09 &  1.28e+01 \\ 
 $\beta$ = 1.00e-01 & SDPLR &-3.1439583e+07  &  9.38e-07 &4.14e-08   & 9.06e-06 &  1.68e+03 \\ 
 $p$ = 2.50e-01 & SketchyCGAL &- &  -&-  &- & - \\ 
 \hline
 $n$ = 5000 & Algorithm1 & -9.5441832e+07&  9.24e-15 & 1.01e-08  & 2.13e-10 &  1.64e+01 \\ 
 $\beta$ = 3.00e-01 & SDPLR &- &  -&-  &- & - \\ 
 $p$ = 2.50e-01 & SketchyCGAL &- &  -&-  &- & - \\ 
 \hline
 $n$ = 5000 & Algorithm1 & -1.5831412e+08&  1.03e-14 & 8.48e-09  & 7.72e-11 &  1.88e+01 \\ 
 $\beta$ = 5.00e-01 & SDPLR &- &  -&-  &- & - \\ 
 $p$ = 2.50e-01 & SketchyCGAL &- &  -&-  &- & - \\ 
 \hline
 $n$ = 5000 & Algorithm1 & -2.2086974e+08&  1.12e-14 & 1.58e-08  & 1.05e-11 &  1.64e+01 \\ 
 $\beta$ = 7.00e-01 & SDPLR &- &  -&-  &- & - \\ 
 $p$ = 2.50e-01 & SketchyCGAL &-3.1589361e+08  &  4.02e-01 &6.44e-01   & 9.99e-01 &  3.60e+03 \\ 
 \hline
 $n$ = 10000 & Algorithm1 & -1.2146659e+08&  1.09e-10 & 1.47e-08  & 6.19e-08 &  5.78e+01 \\ 
 $\beta$ = 1.00e-01 & SDPLR &- &  -&-  &- & - \\ 
 $p$ = 2.50e-01 & SketchyCGAL &- &  -&-  &- & - \\ 
 \hline
 $n$ = 10000 & Algorithm1 & -3.7203424e+08&  1.51e-14 & 1.15e-08  & 6.97e-11 &  7.75e+01 \\ 
 $\beta$ = 3.00e-01 & SDPLR &- &  -&-  &- & - \\ 
 $p$ = 2.50e-01 & SketchyCGAL &- &  -&-  &- & - \\ 
 \hline
 $n$ = 10000 & Algorithm1 & -6.2572380e+08&  1.73e-14 & 1.15e-08  & 6.77e-12 &  6.70e+01 \\ 
 $\beta$ = 5.00e-01 & SDPLR &- &  -&-  &- & - \\ 
 $p$ = 2.50e-01 & SketchyCGAL &- &  -&-  &- & - \\ 
 \hline
 $n$ = 10000 & Algorithm1 & -8.8036592e+08&  1.87e-14 & 2.73e-09  & 6.84e-11 &  7.19e+01 \\ 
 $\beta$ = 7.00e-01 & SDPLR &- &  -&-  &- & - \\ 
 $p$ = 2.50e-01 & SketchyCGAL &- &  -&-  &- & - \\ 
 \hline
 \end{longtable}
\end{footnotesize}
\end{center}

From the results in Table~\ref{SQKR}, we can see that Algorithm 1 is much more efficient than the other two algorithms in all the examples. Algorithms 1 can solve all the instances accurately in a short time. In particular, it can solve each of the SDP problems (SQKE) with $n=10,000$ 
very accurately in about a minute.
SDPLR can only solve the instances of moderate size. As for SketchyCGAL, it reaches the maximum running time in all  the instances except for the first one. Also, the output of Algorithm 1 is always more accurate than  its competitors. Moreover, when $\beta$ is increased, the running time of SDPLR increases dramatically, while the running time of Algorithm 1 becomes even better. This means that Algorithm 1 is more robust than SDPLR. One possible reason is that for SDPLR, the algorithm is based on  ALM. For infeasible methods like ALM and ADMM, their speed is very sensitive to the penalty parameter. Since the best penalty parameter is related to the input data, a different input data may result in significant difference in the running time. However, Algorithm 1 is a feasible method which {does not} need any penalty parameter updating. Thus, Algorithm 1 is more robust than SDPLR and this also shows the advantage of using a feasible method.

%%%%%%%%%%%%%%%%%%%%%%QKP test %%%%%%%%%%%%%%%%%%%%%%%%%%%%%%%%
\subsection{Quadratic knapsack problem.}\label{IntQKP}
In this section, we solve (QKP) approximately by first solving (SQKELR) and then followed by applying our rounding procedure described in subsection~\ref{Rounding}. In order to test the accuracy of
 a feasible solution $d$ of (QKP), we compute the relative gap between $\<C,X\>$ for some optimal solution $Y=\begin{pmatrix}1&x^\top \\x &X\end{pmatrix}$ of (SQKE) and $d^\top Cd$ i.e.,
\begin{equation}\label{Jun_12_1}
{\rm relgap} = \frac{|\<C,X\>-d^\top Cd|}{1+|d^\top Cd|}.
\end{equation}
We will compare our algorithm with a method called DP Heuristic proposed by Fomeni and Letchford in \cite{QKPH}, which is a modified dynamic programming algorithm. {The source codes can be downloaded from \href{https://sites.google.com/view/franklindjeumoufomeni/research/quadratic-knapsack-problem}{https://sites.google.com/view/franklindjeumoufomeni/research/quadratic-knapsack-problem}.} It has been shown in the numerical experiments of \cite{QKPH} that their algorithm can find a feasible solution for (QKP) that is quite close to the exact solution. We generate the dataset in the same way as mentioned in the previous subsection. However, since their codes require that the knapsack capacity $\tau$ to be an integer, we choose $\tau$ to be $\lceil\beta \cdot e^\top a\rceil,$ where $\beta \in \{0.1,0.5\}.$ We also choose the density parameter $p\in\{0.1,0.5\}.$ Although the theoretical rank bound for (SQKELR) is $\lceil \sqrt{2(n+1)}\rceil+2$, we find that the rank of an optimal solution of (SQKELR) is usually less than 20. Therefore, we choose the rank parameter $r:=\min\left\{20,\lceil \sqrt{2(n+1)}\rceil+2\right\}$ to further accelerate Algorithm 1. We set the maximum running time to be 3600 seconds.\\

\begin{center}
\begin{footnotesize}
\begin{longtable}{|c|c|cc|c|l|}
\caption{Comparison between Algorithm1 with rounding procedure and DP Heuristic.}
\label{QKP}
\\
\hline
$n\ |$ $\beta\ |$ $p\ |$ $f_{\rm SDP}$ & algorithm & obj & relgap & time \\ \hline
\endhead
$100| \ 0.1| \ 0.1$ & Algorithm1 & 5524 & 1.49e-01 & 1.25e-01 \\
 6.3497166e+03 & DP Heuristic & 5736 & 1.07e-01 & 1.88e-01 \\
 \hline
 $100| \ 0.1| \ 0.5$ & Algorithm1 & 29283 & 1.45e-02 & 5.56e-02 \\
 2.9708357e+04 & DP Heuristic & 29216 & 1.69e-02 & 1.91e-01 \\
 \hline
 $100| \ 0.5| \ 0.1$ & Algorithm1 & 28535 & 1.48e-02 & 3.67e-02 \\
 2.8957497e+04 & DP Heuristic & 28535 & 1.48e-02 & 3.94e-01 \\
 \hline
 $100| \ 0.5| \ 0.5$ & Algorithm1 & 125252 & 1.06e-03 & 3.98e-02 \\
 1.2538509e+05 & DP Heuristic & 125252 & 1.06e-03 & 3.35e-01 \\
 \hline
 $500| \ 0.1| \ 0.1$ & Algorithm1 & 120312 & 2.41e-02 & 5.09e-01 \\
 1.2321499e+05 & DP Heuristic & 120554 & 2.21e-02 & 7.45e+01 \\
 \hline
 $500| \ 0.1| \ 0.5$ & Algorithm1 & 598541 & 2.38e-03 & 5.67e-01 \\
 5.9996349e+05 & DP Heuristic & 599355 & 1.02e-03 & 7.88e+01 \\
 \hline
 $500| \ 0.5| \ 0.1$ & Algorithm1 & 648454 & 4.31e-03 & 3.16e-01 \\
 6.5124711e+05 & DP Heuristic & 648915 & 3.59e-03 & 9.35e+01 \\
 \hline
 $500| \ 0.5| \ 0.5$ & Algorithm1 & 3165783 & 1.91e-04 & 4.47e-01 \\
 3.1663868e+06 & DP Heuristic & 3166162 & 7.10e-05 & 9.08e+01 \\
 \hline
 $1000| \ 0.1| \ 0.1$ & Algorithm1 & 500226 & 9.65e-03 & 7.21e-01 \\
 5.0505314e+05 & DP Heuristic & 501407 & 7.27e-03 & 1.17e+03 \\
 \hline
 $1000| \ 0.1| \ 0.5$ & Algorithm1 & 2414103 & 4.22e-03 & 5.01e-01 \\
 2.4243016e+06 & DP Heuristic & 2420049 & 1.76e-03 & 1.19e+03 \\
 \hline
 $1000| \ 0.5| \ 0.1$ & Algorithm1 & 2574765 & 1.64e-03 & 4.91e-01 \\
 2.5789825e+06 & DP Heuristic & 2575821 & 1.23e-03 & 1.34e+03 \\
 \hline
 $1000| \ 0.5| \ 0.5$ & Algorithm1 & 12428992 & 2.10e-03 & 7.54e-01 \\
 1.2455036e+07 & DP Heuristic & 12453855 & 9.48e-05 & 1.34e+03 \\
 \hline
 $5000| \ 0.1| \ 0.1$ & Algorithm1 & 12644311 & 1.07e-03 & 6.92e+00 \\
 1.2657897e+07 & DP Heuristic & - & - & - \\
 \hline
 $5000| \ 0.1| \ 0.5$ & Algorithm1 & 60961427 & 1.77e-03 & 5.51e+00 \\
 6.1069318e+07 & DP Heuristic & - & - & - \\
 \hline
 $5000| \ 0.5| \ 0.1$ & Algorithm1 & 62978380 & 5.24e-04 & 6.12e+00 \\
 6.3011370e+07 & DP Heuristic & - & - & - \\
 \hline
 $5000| \ 0.5| \ 0.5$ & Algorithm1 & 315829856 & 1.99e-05 & 6.95e+00 \\
 3.1583613e+08 & DP Heuristic & - & - & - \\
 \hline
 $10000| \ 0.1| \ 0.1$ & Algorithm1 & 48436428 & 8.39e-04 & 2.44e+01 \\
 4.8477060e+07 & DP Heuristic & - & - & - \\
 \hline
 $10000| \ 0.1| \ 0.5$ & Algorithm1 & 243618278 & 8.80e-04 & 2.13e+01 \\
 2.4383258e+08 & DP Heuristic & - & - & - \\
 \hline
 $10000| \ 0.5| \ 0.1$ & Algorithm1 & 251900539 & 1.04e-04 & 2.68e+01 \\
 2.5192683e+08 & DP Heuristic & - & - & - \\
 \hline
 $10000| \ 0.5| \ 0.5$ & Algorithm1 & 1251839038 & 1.82e-04 & 2.68e+01 \\
 1.2520674e+09 & DP Heuristic & - & - & - \\
 \hline
 \end{longtable}
\end{footnotesize}
\end{center}

From the results in Table~\ref{QKP}, we can see that the output of DP Heuristic is slightly more accurate than the SDP model with rounding procedure. This is reasonable because DP Heuristic is a dynamic programming algorithm with a lot of delicated enhancements. However, the overall accuracy of the SDP model and DP Heuristic is at the same level. Moreover, Algorithm 1 is much more efficient than DP Heuristic. It has been mentioned in \cite{QKPH} that the computational complexity of DP Heuristic is $O(n^2\tau).$ Since we choose $\tau=\lceil \beta\cdot e^\top a\rceil=\Theta(n),$ the complexity becomes $O(n^3),$ which becomes extremely slow when $n$ is large. This can also be seen from Table~\ref{QKP}. When $n=100,$ both DP Heuristic and Algorithm 1 can terminate within a second. However, when $n=1000,$ DP Heuristic needs more than 1000 seconds to return a solution while the running time of our algorithm is still less than a second. In this case, our algorithm is more than 1000 times faster than DP Heuristic. When $n=5000$ or $10000,$ DP Heuristic reaches the maximum running time but our algorithm can terminate within 30 seconds. 

\subsection{Testing Algorithm 1 for problems with non-regular points.}\label{SigQKP}
Although Algorithm 1 is able to verify the global optimality of a non-regular point and escape from it if it is not optimal, we {have not} met any non-regular points in the numerical experiments of the previous subsections. One reason is that for a generic input data $\(a,\tau\)\in \R^n\times \R$ that satisfies Assumption~\ref{ass1}, $\Ka$ is regular everywhere except for a trivial point (see Corollary~\ref{regular1}). Also, a trivial point implies that we {do not} select any item in the knapsack problem, which is unlikely to happen in practice. In order to test the ability of Algorithm 1 to deal with non-regular points, we generate a special dataset $\(C',a',\tau'\)\in \S^n\times \R^n\times \R$ such that the optimal solution of (SQKE) contains a non-regular point, i.e., an integer solution. For an even number $n\in \mathbb{N}^+,$ we first generate a random data $\(C,a,\tau\)\in \S^n\times \R^n\times \R$ using the same method as in subsection~\ref{NumerQKP}. We then modify the data $\(C,a,\tau\)$ to get a new data $\(C',a',\tau'\)\in \S^n\times \R^n\times \R$ as follows:
For any $i,j\in [n]$,
\begin{equation}\label{Jun_18_1}
C'_{ij}:=\begin{cases}C_{ij}&{i\ {\rm and}\ j\ {\rm are\ even}} \\ 0 & {\rm otherwise}\end{cases},\quad a'_i:=a_{2\lceil i/2\rceil},\quad \tau'=\frac{1}{2}e^\top a'.
\end{equation}

Consider problem (SQKE) and (SQKELR) with input data $\(C',a',\tau'\).$ Define two vectors $v_1,v_2\in \{0,1\}^n$ such that the index of nonzero entries of $v_1$ and $v_2$ are odd numbers and even numbers respectively. From (\ref{Jun_18_1}), we have that $a'^\top v_1=a'^\top v_2=\tau'.$ Thus, both $v_1e_1^\top$ and $v_2e_1^\top$ are inside $\N_{n,r}$ and so they are non-regular points of $\Ka.$ Also, it is easy to verify that $R=v_2e_1^\top$ is a global optimal solution of (SQKELR) and $Y=\(1\ v_2^\top\)^\top \(1\ v_2^\top\)$ is a global optimal solution of (SQKE). Actually, if for any $i=2k\in [n],$ there exists $j=2\ell\in [n]$ such that $C'_{ij}>0,$ then $R=v_2e_1^\top$ and $Y=\(1\ v_2^\top\)^\top \(1\ v_2^\top\)$ are the only global optimal solution of (SQKELR) and (SQKE) respectively. In this numerical experiment, we choose $n\in\{1000,2000,5000,7000,10000\},$ and $p=0.25.$ We choose $r=3$ because all the instances have a rank-one optimal solution. We generate the initial point $R_0$ by adding a small perturbation to the non-regular non-optimal point $v_1e_1^\top.$ The {\sc Matlab} code for the initialization is: \texttt{R0 = v1*[1,0,0]+rand(n,3)/(1000*n);}  In this case, the escaping procedure of Algorithm 1 is likely to be triggered at the beginning of the algorithm. To demonstrate the usefulness of the rounding procedure in Algorithm 1, we also tested Algorithm 1 without the rounding procedure, i.e., the Riemannian gradient descent method with BB step and non-monotone line search mentioned in subsection~\ref{RieRound}. We call this method RieBB.

\begin{center}
\begin{footnotesize}
\begin{longtable}{|c|c|cccc|l|}
\caption{Comparison between Algorithm 1 and the Riemannian optimization method, RieBB without the rounding procedure, for quadratic knapsack SDPs with non-regular optimal solutions.}
\label{SQKSIG}
\\
\hline
problem  & Algorithm & obj & Rp & Rd& pdgap & time \\ \hline
\endhead

$n$ = 1000 & Algorithm1 & -3.1407370e+06&  0.00e+00 & 0.00e+00  & 1.31e-16 &  3.42e-01 \\ 
 $p$ = 0.25 & RieBB & -3.1407370e+06&  2.99e-10 & 0.00e+00  & 5.19e-12 &  4.78e+00 \\ 
 \hline
 $n$ = 2000 & Algorithm1 & -1.2725874e+07&  0.00e+00 & 0.00e+00  & 1.30e-16 &  7.57e-01 \\ 
 $p$ = 0.25 & RieBB & -1.2725874e+07&  9.64e-11 & 0.00e+00  & 3.70e-14 &  1.03e+00 \\ 
 \hline
 $n$ = 5000 & Algorithm1 & -7.8879089e+07&  2.22e-16 & 0.00e+00  & 2.09e-16 &  4.75e+00 \\ 
 $p$ = 0.25 & RieBB & -7.8879089e+07&  3.03e-10 & 3.08e-15  & 1.90e-12 &  6.43e+00 \\ 
 \hline
 $n$ = 7000 & Algorithm1 & -1.5456007e+08&  1.78e-15 & 3.86e-18  & 1.57e-15 &  9.95e+00 \\ 
 $p$ = 0.25 & RieBB & -1.5456007e+08&  1.35e-12 & 0.00e+00  & 3.86e-16 &  1.12e+02 \\ 
 \hline
 $n$ = 10000 & Algorithm1 & -3.1552221e+08&  8.88e-15 & 0.00e+00  & 1.26e-15 &  2.00e+01 \\ 
 $p$ = 0.25 & RieBB & -3.1552221e+08&  7.03e-12 & 0.00e+00  & 1.13e-15 &  4.51e+01 \\ 
 \hline
 \end{longtable}
\end{footnotesize}
\end{center}

From the results in Table~\ref{SQKSIG}, we can see that Algorithm 1 can solve all of the instances efficiently. Also, note that the KKT residues of the solutions are close to the rounding error. This is because our test problems has an integer solution $v_2e_1^\top,$ which can be obtained exactly by the rounding step in Algorithm 1. The performance of RieBB is unstable. For the instance with $n=7000$, it is more than 10 times slower than Algorithm 1. This is because the Riemannian optimization method, RieBB, may suffer from degeneracy issue when the iterations approach a non-regular point, and that may slow down the algorithm. However, this will not happen in Algorithm 1 because the rounding procedure allows the algorithm to detect a non-regular point in advance and escape from it if it is non-optimal.

\subsection{Testing Algorithm 1 for large sparse problems.}\label{Largesparse}
It has been verified in the previous subsections that Algorithm 1 is more efficient than other SDP solvers that can also utilise the low rank property of a linear SDP problem. In this subsection, we move on to test some large instances of (SQKE) where $n\geq 10000$ and $C$ is a sparse random matrix. The matrix $C$ is generated in the same way as in subsection~\ref{NumerQKP} with density $p=\log(n)/n.$ We choose $\beta = 0.5$ and $n\in\{10^4,2\cdot10^4,5\cdot 10^4,10^5,2\cdot 10^5,5\cdot 10^5,10^6\}.$ We also set $r=20$ as in subsection~\ref{IntQKP}. Note that when $n$ is very large, computing the dual residue Rd (see (\ref{KKTres})) will be extremely slow because it requires the full eigenvalue decomposition of a large dense matrix ${\bf S},$ which is the dual slack defined in (\ref{dslack}). However, due to the special structure of ${\bf S},$ i.e., it can be written as the sum of a sparse matrix and a low rank matrix, we can use the solver lobpcg developed by Knyazev in \cite{lobpcg} to compute the smallest eigenvalue of ${\bf S}$ efficiently. {The source codes can be downloaded from \href{https://github.com/lobpcg/blopex}{https://github.com/lobpcg/blopex}.} We redefine the dual residue ${\rm Rd}:=\frac{\max\{0,-\lambda_{\min}\({\bf S}\)\}}{1+\|{\bf S}\|_F}$, 
{so that only the smallest eigenvalue of ${\bf S}$ needs to be computed.}
It is easy to see that ${\rm Rd}=0$ if and only if ${\bf S}\in \S^{n+1}_+.$

\begin{center}
\begin{footnotesize}
\begin{longtable}{|c|ccccc|l|}
\caption{Testing Algorithm 1 for quadratic knapsack SDPs with large sparse profit matrix $C.$}
\label{SQKLARGE}
\\
\hline
problem  & obj & relgap& Rp & Rd& pdgap & time \\ \hline
\endhead
$n$ = 10000  & -2.7522454e+06&  1.86e-02&  6.76e-14 & 5.97e-08  & 7.11e-08 &  7.29e+00 \\ 
 \hline
 $n$ = 20000  & -5.8646881e+06&  1.88e-02&  4.32e-11 & 9.77e-08  & 5.59e-09 &  1.25e+01 \\ 
 \hline
 $n$ = 50000  & -1.5887212e+07&  1.74e-02&  2.37e-12 & 5.98e-08  & 3.97e-07 &  2.68e+01 \\ 
 \hline
 $n$ = 100000  & -3.3509978e+07&  1.81e-02&  3.72e-12 & 3.52e-08  & 2.81e-08 &  4.16e+01 \\ 
 \hline
 $n$ = 200000  & -7.1116522e+07&  2.41e-02&  2.44e-11 & 1.54e-09  & 4.50e-08 &  7.98e+01 \\ 
 \hline
 $n$ = 500000  & -1.8974901e+08&  2.64e-02&  7.80e-09 & 4.90e-09  & 1.57e-07 &  3.39e+02 \\ 
 \hline
 $n$ = 1000000  & -3.9832359e+08&  2.79e-02&  2.47e-10 & 1.51e-09  & 2.40e-08 &  1.26e+03 \\ 
 \hline
 \end{longtable}
\end{footnotesize}
\end{center}
From the results in Table~\ref{SQKLARGE}, we can see that Algorithm 1 can solve all the problems successfully. For the instance of size $n=10^6,$ Algorithm 1 can compute a solution with high accuracy in about 20 minutes. Also, it should be noted that for a sparse profit matrix $C,$ the relative gap between the SDP bound and the objective value of the rounding solution is close to $2\%,$ which is considered to be very good for an NP-hard integer programming problem.

\section{Conclusion.}\label{Conclusion}
In this paper, we study one of the SDP relaxations of the quadratic knapsack problem. We consider the low rank formulation of this SDP problem to reduce the problem's dimensionality. In order to solve the low rank problem efficiently, we explore the geometric properties of its feasible region $\Ka$, which is an algebraic variety. We prove that this algebraic variety is smooth everywhere for a generic input data except for a trivial point. We also study the local geometric properties of $\Ka$ at non-regular points on this algebraic variety.  In order to find the global optimal solution of this non-convex problem, we derive a new-rank bound for these two problems to be equivalent. We also prove that there is no spurious local minima under some rank condition without using any regularity condition. All these good properties allow us to design a feasible method based on Riemannian optimization, which can also handle non-regular points. We conduct numerical experiments to verify the efficiency and robustness of our feasible method. We also apply a rounding procedure to the optimal solution of the SDP relaxation to get a feasible solution of the quadratic knapsack problem. Numerical experiments show that our feasible method with rounding strategy is much more efficient than another heuristic based on dynamic programming. Traditional feasible methods focus on simple smooth manifolds like oblique manifold and Stiefel manifold and thus their application range is quite limited. However, this work demonstrates the possibility of extending a Riemannian optimization framework to algebraic varieties. {It would be interesting to generalize this idea for solving more general SDP problems with more complicated constraints. We leave this for future work.}

%%%%%%%%%%%%%%%%%%%%%%%%%%%%Reference%%%%%%%%%%%%%%%%%%%%%%%%%%%%%%%

\newpage
\appendix
\section{Useful auxiliary results}
\subsection{A corollary of the implicit function theorem}
\begin{lem}\label{implicit}
Let $F:\R^{n}\times \R^{m}\rightarrow \R^{k}$ be a continuously differentiable mapping in a neighbourhood of $\( x_0,y_0\)\in \R^{n}\times \R^{m}.$ Suppose $F(x_0,y_0)=0$ and $\D_y F(x_0,y_0):\R^m\rightarrow \R^k$ is a surjective linear mapping. Then there exists $\delta>0$ and a continuously differentiable mapping $y:B_\delta(x_0)\rightarrow \R^k$ such that $y(x_0)=y_0$ and $F(x,y(x))=0.$ Here, $B_\delta(x_0):=\left\{ x\in \R^n:\ \dist(x,x_0)<\delta\right\}.$
\end{lem}
\begin{proof}
\pf
Since $\D_y F(x_0,y_0)$ is surjective, we have that $m\geq k.$ If $m=k,$ then Lemma~\ref{implicit} directly follows from the implicit function theorem and the mapping $y$ is unique. If $m>k$, we may find a linear mapping $\A:\R^m\rightarrow \R^{m-k}$ such that the linear operator $\triangle y\rightarrow \(\D_y F(x_0,y_0)[\triangle y],\A(\triangle y)\)$ is a bijective mapping from $\R^m$ to $\R^m.$ Thus, we can define a new mapping $\widehat{F}:\R^{n}\times \R^{m}\rightarrow \R^m$ such that $\widehat{F}(x,y):=\(F(x,y),\A(y)-\A(y_0)\)$ and apply the implicit function theorem to $\widehat{F}.$  
\end{proof}

\subsection{(\ref{escSDP}) and (\ref{escSDPd}) are strictly feasible}
\begin{lem}\label{exikkt} 
Suppose $(a,\tau)$ satisfies Assumption~\ref{ass1}. Then both (\ref{escSDP}) and (\ref{escSDPd}) have a strictly feasible solution.
 \end{lem}
 \begin{proof}
 \pf
 It is easy to see that the feasibility of (\ref{escSDPd}) is strictly satisfied by choosing $\beta$ to be sufficiently negative. We only have to find a strictly feasible solution for (\ref{escSDP}). \\
 
 \noindent{\bf Case 1.} $v\neq 0.$\\
 In this case, $\sigma_v=1.$ Let $X_1:=\frac{ee^\top}{n}.$ We have that $X_1\succeq 0,$ $\Tr(X_1)=1.$ Also,
 \begin{equation}\label{May_24_4}
 \<aa^\top-\sigma_v\tau\cdot\dd\(a\circ d\),X_1\>=\frac{\(e^\top a\)^2-\tau d^\top a}{n}>\frac{\(e^\top a\)^2-\tau e^\top a}{n}>0,
 \end{equation}
 where the last inequality comes from the fact that $e^\top a>\tau.$ Because $v\neq 0,$ there exists $i\in [n]$ such that $d_i=1.$ For $t\in [0,1],$ define $Y_t=(1-t)\dd\(e_i\)+\frac{t}{n}I.$ We have that $Y_t\succeq 0$ and $\Tr(Y_t)=1.$ Also, 
 \begin{equation}\label{May_24_5}
 \<aa^\top-\sigma_v\tau\cdot\dd\(a\circ d\),Y_0\>=a_i^2-\tau a_i<0,
 \end{equation}
 where the last inequality comes from Assumption~\ref{ass1}. From continuity, there exists $t\in (0,1)$ such that $ \<aa^\top-\sigma_v\tau\cdot\dd\(a\circ d\),Y_t\><0$ and $Y_t\succ 0.$ From (\ref{May_24_4}), there exists $\beta\in \(0,1\)$ such that $\beta X_1+(1-\beta)Y_t\succ 0,$ $\<I,\beta X_1+(1-\beta)Y_t\>=1$  and 
   \begin{equation}\label{May_24_6}
 \<aa^\top-\sigma_v\tau\cdot\dd\(a\circ d\),\beta X_1+(1-\beta)Y_t\>=0.
 \end{equation}
 Therefore, $\beta X_1+(1-\beta)Y_t$ is a strict feasible solution of (\ref{escSDP}).\\
 
 \noindent{\bf Case 2.} $v=0.$\\
 In this case, $\sigma_v=-1$ and $aa^\top-\sigma_v\tau\cdot\dd\(a\circ d\)=aa^\top-\tau\dd\(a\).$ {We} can construct a strictly feasible solution in the same way as in Case 1.
 \end{proof}


\begin{thebibliography}{99}
	\addtolength{\baselineskip}{-1ex}
	\bibitem{PR} Absil, P. A. and Malick, J. (2012). Projection-like retractions on matrix manifolds. SIAM Journal on Optimization, 22(1), 135-158.
		
	\bibitem{rank1} Barvinok, A. I. (1995). Problems of distance geometry and convex properties of quadratic maps. Discrete \& Computational Geometry, 13(2), 189-202.
	
	\bibitem{BB} Barzilai, J. and Borwein, J. M. (1988). Two-point step size gradient methods. IMA Journal of Numerical Analysis, 8(1), 141-148.
	
	\bibitem{Boumal0} Bhojanapalli, S., Boumal, N., Jain, P. and Netrapalli, P. (2018, July). Smoothed analysis for low-rank solutions to semidefinite programs in quadratic penalty form. In Conference On Learning Theory (pp. 3243-3270). PMLR.
	
	\bibitem{exactQKP} Billionnet, A. and Soutif, É. (2004). An exact method based on Lagrangian decomposition for the 0–1 quadratic knapsack problem. European Journal of Operational Research, 157(3), 565-575.
	
	\bibitem{Perturbation} Bonnans, J. F. and Shapiro, A. (2013). Perturbation analysis of optimization problems. Springer Science \& Business Media.
	
	\bibitem{intromani} Boumal, N. (2020). An introduction to optimization on smooth manifolds. Available online.
	
	\bibitem{manopt} Boumal, N., Mishra, B., Absil, P. A. and Sepulchre, R. (2014). Manopt, a {\sc Matlab} toolbox for optimization on manifolds. Journal of Machine Learning Research, 15(1), 1455-1459.
	
	\bibitem{Boumal1} Boumal, N., Voroninski, V. and Bandeira, A. (2016). The non-convex Burer-Monteiro approach works on smooth semidefinite programs. Advances in Neural Information Processing Systems, 29.
	
	\bibitem{Boumal2} Boumal, N., Voroninski, V. and Bandeira, A. S. (2020). Deterministic guarantees for Burer‐Monteiro factorizations of smooth semidefinite programs. Communications on Pure and Applied Mathematics, 73(3), 581-608.
	
	\bibitem{BMmaxcut} Burer, S. and Monteiro, R. D. (2001). A projected gradient algorithm for solving the maxcut SDP relaxation. Optimization Methods and Software, 15(3-4), 175-200.
	
	\bibitem{BM1} Burer, S. and Monteiro, R. D. (2003). A nonlinear programming algorithm for solving semidefinite programs via low-rank factorization. Mathematical Programming, 95(2), 329-357.
	
	\bibitem{BM2} Burer, S. and Monteiro, R. D. (2005). Local minima and convergence in low-rank semidefinite programming. Mathematical Programming, 103(3), 427-444.	
	
	\bibitem{exactlarge} Caprara, A., Pisinger, D. and Toth, P. (1999). Exact solution of the quadratic knapsack problem. INFORMS Journal on Computing, 11(2), 125-137.
	
	\bibitem{MG1} Fallat, S. M. and Hogben, L. (2007). The minimum rank of symmetric matrices described by a graph: a survey. Linear Algebra and its Applications, 426(2-3), 558-582.
	
	\bibitem{VLSI} Ferreira, C. E., Martin, A., de Souza, C. C., Weismantel, R. and Wolsey, L. A. (1996). Formulations and valid inequalities for the node capacitated graph partitioning problem. Mathematical Programming, 74(3), 247-266.
	
	\bibitem{QKPH} Fomeni, F. D. and Letchford, A. N. (2014). A dynamic programming heuristic for the quadratic knapsack problem. INFORMS Journal on Computing, 26(1), 173-182.
	
	\bibitem{rank2} Friedland, S. and Loewy, R. (1976). Subspaces of symmetric matrices containing matrices with a multiple first eigenvalue. Pacific Journal of Mathematics, 62(2), 389-399.
	
	\bibitem{RieBBrank} Gao, B. and Absil, P. A. (2022). A Riemannian rank-adaptive method for low-rank matrix completion. Computational Optimization and Applications, 81(1), 67-90.
	
	\bibitem{QKP} Gallo, G., Hammer, P. L. and Simeone, B. (1980). Quadratic knapsack problems. In Combinatorial optimization (pp. 132-149). Springer, Berlin, Heidelberg.
	
	\bibitem{maxcut} Goemans, M. X. and Williamson, D. P. (1995). Improved approximation algorithms for maximum cut and satisfiability problems using semidefinite programming. Journal of the ACM (JACM), 42(6), 1115-1145.	
	
	\bibitem{SQK2} Helmberg, C., Rendl, F. and Weismantel, R. (2000). A semidefinite programming approach to the quadratic knapsack problem. Journal of Combinatorial Optimization, 4(2), 197-215.
	
	\bibitem{MGn2} Hogben, L. and Van Der Holst, H. (2007). Forbidden minors for the class of graphs G with $\xi (G)\leq 2$. Linear Algebra and its Applications, 423(1), 42-52.
	
	\bibitem{MG2} Hogben, L. (2010). Minimum rank problems. Linear Algebra and its Applications, 432(8), 1961-1974.
	
	\bibitem{RieBB} Iannazzo, B. and Porcelli, M. (2018). The Riemannian Barzilai–Borwein method with nonmonotone line search and the matrix geometric mean computation. IMA Journal of Numerical Analysis, 38(1), 495-517.
	
	\bibitem{compiler} Johnson, E. L., Mehrotra, A. and Nemhauser, G. L. (1993). Min-cut clustering. Mathematical Programming, 62(1), 133-151.
	
	\bibitem{lobpcg} Knyazev, A. V. (2001). Toward the optimal preconditioned eigensolver: Locally optimal block preconditioned conjugate gradient method. SIAM Journal on Scientific Computing, 23(2), 517-541.
	
	\bibitem{mosek} {\tt https://www.mosek.com/}.
	
	\bibitem{rank3} Pataki, G. (1998). On the rank of extreme matrices in semidefinite programs and the multiplicity of optimal eigenvalues. Mathematics of Operations Research, 23(2), 339-358.
	
	\bibitem{matrixcook} Petersen, K. B. and Pedersen, M. S. (2008). The matrix cookbook. Technical University of Denmark, 7(15), 510.
	
	\bibitem{hardKP} Pisinger, D. (2005). Where are the hard knapsack problems? Computers \& Operations Research, 32(9), 2271-2284.
	
	\bibitem{SQK3} Pisinger, D. (2007). The quadratic knapsack problem—a survey. Discrete Applied Mathematics, 155(5), 623-648.
	
	\bibitem{BB1} Raydan, M. (1997). The Barzilai and Borwein gradient method for the large scale unconstrained minimization problem. SIAM Journal on Optimization, 7(1), 26-33.
	
	\bibitem{GEP} Tang, T. and Toh, K. C. (2021). Solving graph equipartition SDPs on an algebraic variety. arXiv preprint arXiv:2112.04256.
	
	\bibitem{SDPTTT} Toh, K. C., Todd, M. J. and Tütüncü, R. H. (1999). SDPT3—a {\sc Matlab} software package for semidefinite programming, version 1.3. Optimization Methods and Software, 11(1-4), 545-581.
	
	\bibitem{SDPT3} Tütüncü, R. H., Toh, K. C. and Todd, M. J. (2003). Solving semidefinite-quadratic-linear programs using SDPT3. Mathematical Programming, 95(2), 189-217.
	
	\bibitem{knapcom} Wang, H., Kochenberger, G. and Glover, F. (2012). A computational study on the quadratic knapsack problem with multiple constraints. Computers \& Operations Research, 39(1), 3-11.
	
	\bibitem{WenYin} Wen, Z. and Yin, W. (2013). A feasible method for optimization with orthogonality constraints. Mathematical Programming, 142(1), 397-434.
	
	\bibitem{numerical} Wright, S. and Nocedal, J. (1999). Numerical optimization. Springer Science, 35(67-68), 7.
	
	\bibitem{SDPNALp} Yang, L., Sun, D. and Toh, K. C. (2015). SDPNAL+: a majorized semismooth Newton-CG augmented Lagrangian method for semidefinite programming with nonnegative constraints. Mathematical Programming Computation, 7(3), 331-366.
	
	\bibitem{CGAL} Yurtsever, A., Tropp, J. A., Fercoq, O., Udell, M. and Cevher, V. (2021). Scalable semidefinite programming. SIAM Journal on Mathematics of Data Science, 3(1), 171-200.
	
	\bibitem{NR} Zhang, R. (2020). Newton retraction as approximate geodesics on submanifolds. arXiv preprint arXiv:2006.14751.
	
	\bibitem{SDPNAL} Zhao, X. Y., Sun, D. and Toh, K. C. (2010). A Newton-CG augmented Lagrangian method for semidefinite programming. SIAM Journal on Optimization, 20(4), 1737-1765.
	
\end{thebibliography}
\end{document}